\documentclass[11pt,reqno,oneside]{amsart}

\usepackage{graphicx}

\usepackage{amssymb}
\usepackage{epstopdf}

\usepackage[a4paper, total={6in, 9.1in}]{geometry}

\usepackage{amsmath,amsfonts,amsthm,mathrsfs,amssymb,cite}

\usepackage{subfigure,tikz}
\usetikzlibrary{patterns}
\usetikzlibrary{hobby}
\usepackage{framed}

\usepackage{enumerate}
\usepackage{hyperref}
\usepackage{xcolor}
\hypersetup{
  colorlinks,
  linkcolor={blue!80!black},
  urlcolor={blue!80!black},
  citecolor={blue!80!black}
}
\usepackage[capitalize,noabbrev]{cleveref}

\numberwithin{equation}{section}

\newtheorem{Theorem}{Theorem}[section]
\newtheorem{Lemma}[Theorem]{Lemma}
\newtheorem{Proposition}[Theorem]{Proposition}

\newtheorem{Remark}[Theorem]{Remark}
\numberwithin{equation}{section}

\def\bbr{{\mathbb R}}
\def\bbn{{\mathbb{N}}}
\newcommand{\mi}{\mathrm{i}}
\newcommand{\bx}{\mathbf{x}}
\newcommand{\coma}{\; , \;}
\newcommand{\dx}[1]{\, \mathrm{d} #1}

\title[Generating  field concentration  via surface transmission resonance]{Generating customized field concentration via virtual surface transmission resonance}

\author{Yueguang Hu}
\address{Department of Mathematics, City University of Hong Kong, Kowloon, Hong Kong, China}
\email{yueghu2-c@my.cityu.edu.hk}

\author{Hongyu Liu}
\address{Department of Mathematics, City University of Hong Kong, Kowloon, Hong Kong, China}
\email{hongyu.liuip@gmail.com, hongyliu@cityu.edu.hk}

\author{Xianchao Wang}
\address{School of Mathematics, Harbin Institute of Technology, Harbin, China}
\email{xcwang90@gmail.com}

\author{Deyue Zhang}
\address{School of Mathematics, Jilin University,  Changchun,  China}
\email{dyzhang@jlu.edu.cn}

\begin{document}

\begin{abstract}
In this paper, we develop a mathematical framework for generating strong customized field concentration locally around the inhomogeneous medium inclusion via surface transmission resonance.
The purpose of this paper is twofold.
Firstly, we show that for a given inclusion embedded in an otherwise uniformly homogeneous background space, we can design an incident field  to generate strong localized field concentration  at any specified places around the inclusion.  The aforementioned customized field concentration is crucially reliant on the peculiar  spectral and geometric patterns of certain transmission eigenfunctions.
Secondly, we prove the existence of a sequence of transmission eigenfunctions for a specific wavenumber and they exhibit distinct surface resonant behaviors, accompanying  strong surface-localization and surface-oscillation properties. These eigenfunctions as the surface transmission resonant modes fulfill the requirement for generating the field concentration.

\end{abstract}
\maketitle

\medskip
	


\section{Introduction}

\subsection{Statement of main results}
In this paper, we are concerned with artificially generating field concentration via surface transmission resonance. To begin with, we present the mathematical formulation of the transmission problem and discuss the major findings mainly focusing on the mathematics.

We consider the time-harmonic wave scattering problem arising from  an incident field $u^i$ and a bounded inhomogeneous medium inclusion $(\Omega; \sigma,\tau)$ in $\bbr^N,\,  N=2,\, 3$, which is the mathematical model of our study.
Here, $\Omega$ denotes the  support of the inhomogeneity of the medium  with a connected complement in $\bbr^N$ and   $\sigma,\tau$ are the related material parameters. By choosing appropriate physical units, we may assume that the material parameters by normalization are set as
\begin{equation*}
\tilde{\sigma}(\bx) = \chi_{\bbr^N\backslash \overline{\Omega}}+\sigma(\bx)\chi_{\Omega} \coma \tilde{\tau}(\bx) = \chi_{\bbr^N\backslash \overline{\Omega}}+\tau(\bx)\chi_{\Omega} \coma \bx \in \bbr^N ,
\end{equation*}
where $\sigma(\bx)$ and $\tau(\bx)$ are two  $L^\infty$ functions in $\Omega$.  $\chi_A(\bx)$ is the indicator function such that  $\chi_A (\bx) =1$ if $\bx \in A$ and $\chi_A (\bx) =0 $ otherwise.
Let $u^i$ be an entire solution to homogeneous Helmholtz equation
\begin{equation}\label{incident}
\Delta u^i(\mathbf{x}) +k^2u^i(\mathbf{x}) =0  \coma \mathbf{x} \in\bbr^N,
\end{equation}
where $k \in \bbr_+$ represents the (normalized) wavenumber of wave propagation. The incident field $u^i$ impinging on the inclusion $\Omega$ gives rise to the scattered field $u^s$. Let $u:= u^i +u^s$ denotes the total field.
The forward scattering problem is described by the following Helmholtz system,
\begin{equation}\label{sys1}
\begin{cases}
 \displaystyle \nabla \cdot \left(\tilde{\sigma} \nabla u \right)  + k^2 \tilde{\tau} u = 0 &\text{in} \; \bbr^N, \medskip \\
u = u^i +u^s  &\text{in} \; \bbr^N, \medskip\\
\displaystyle \lim\limits_{r \rightarrow \infty} r^\frac{N-1}{2}\left(\partial_r u^s- \mi k  u^s\right) =  0 \coma &r= |\bx|,
\end{cases}
\end{equation}
where $\mathrm{i} = \sqrt{-1}$ is the imaginary unit and $\partial_r u = \hat{\bx} \cdot \nabla u$ with $\hat{\bx}:= \bx / |\bx| \in \mathbb{S}^{N-1}  $.
The last limit in \eqref{sys1} characterizes the outgoing nature of the scattered  field,  known as the Sommerfeld radiation condition.
We also allow the case of absorption, represented by the material parameters
with nonzero imaginary parts. In this paper, we assume
\begin{equation}\label{parameter}
\Im \sigma(\bx) \leqslant 0 \coma \Im \tau(\bx) \geqslant 0 \coma \Re \sigma (\bx) \geqslant \gamma_0 \coma \bx \in \Omega,
\end{equation}
for some positive constant $\gamma_0$.
Then the system \eqref{sys1} is well-posed,  and in particular, there exists a unique solution $u \in H_{loc}^2(\bbr^N)$ \cite{Colton2019}. Furthermore, the total field  has the following asymptotic expansion:
\begin{equation*}
u(\bx) = u^i (\bx) + \frac{e^{\mi kr}}{r^{(N-1)/2}}u_{\infty}(\hat{\bx}) +\mathcal{O}\left(\frac{1}{r^{(N+1)/2}}\right)  \quad \textrm{as} \quad r \rightarrow +\infty,
\end{equation*}
which holds uniformly for all directions $\hat{\bx} \in \mathbb{S}^{N-1}$. The complex-value function $u_\infty$, defined on the unit sphere $\mathbb{S}^{N-1}$, is known as the far field  pattern and encodes the scattering information
caused by the perturbation of the incident field $u^i$ due to the scatterer $(\Omega;\sigma,\tau)$.

In this paper, we are concerned with artificially generating strong field concentration in a customized way. Specifically, we show that by properly choosing an incident field $u^i$, one can make the degree of field concentration, namely $\|\nabla u\|_{L^\infty}$, to be arbitrarily
large at a customized exterior neighborhood around $\partial \Omega$. The main result is presented as follows.

\begin{Theorem}\label{11}
Consider the scattering problem \eqref{sys1} with a $C^{1,1}$ inclusion $\Omega$ in $\bbr^N,N=2,3$. We denote by $\Gamma \subset \partial \Omega$ any subset of the boundary $\partial \Omega$. Then  the gradient of the total field $\nabla u$ can blow up in the following sense: for any given large number $M>0$, there exists an incident field $u^i$ impinging on the inclusion $\Omega$ such that
\begin{equation*}
 \|\nabla u\|_{L^\infty (\Gamma_e(\Omega,\epsilon))} \geqslant M(u^i),
\end{equation*}
where $M$ is only dependent of the incident field $u^i$ and $\Gamma_e(\Omega,\epsilon)= \{x |\; \mathrm{dist}(x,\Gamma) \leqslant \epsilon, x \notin \overline{\Omega} \}$ is the exterior of the inclusion $\Omega$ along $\Gamma$ for any given parameter $\epsilon = o(1)$.
\end{Theorem}

The primary idea for generating field concentration involves introducing a virtual field concentration generator near the inhomogeneous medium inclusion (see Figure \ref{omega}).
Subsequently, we can design a tailored incident field to activate such a generator and achieve the desired field concentration.
A field concentration generator is referred to as a transmission eigenmode ball where the wave field tends to concentrate around the boundary.
As the activated generator approaches to the inclusion, it produces a significant potential difference in the vicinity around the inclusion that lead to the blowup of gradient of the underlying wave field.
It should be noted that the construction of the incident field $u^i$ is dependent on the distance parameter $\epsilon$ and the geometric location of the boundary part $\Gamma$. Therefore, the dependence of $M$ on the incident field $u^i$ also signifies its dependence on $\epsilon$ and $\Gamma$.
Furthermore, there are no restrictive requirements regarding the occurrence locations of field concentration. In fact, we can achieve field concentration in any specified location (even multiple locations) around the inclusion by introducing field concentration generators over there.

\begin{figure}[htbp!]
\centering
\begin{tikzpicture}
\coordinate  (O) at (0,0);
\filldraw [thick, fill =gray]  plot[smooth, tension=.7] coordinates {(4.4,2.75) (3.3,3.3) (2.3,3.08) (0.88,2.75) (0.55,1.65) (-0.55,0) (0,-2.2)(1.65,-2.75) (4.4,-2.2) (3.85,-0.55) (5.5,1.1) (4.4,2.75)};

\node at (2.5,-1.2) {$\Omega$};
\draw[black,densely dashed]  (4.15,-0.65) circle (0.25);
\draw[black,densely dotted]  (4.15,-0.65) circle (0.2);
\draw[black,densely dotted]  (4.15,-0.65) circle (0.15);
\draw[black,densely dotted]  (4.15,-0.65) circle (0.1);
\draw[black,densely dotted]  (4.15,-0.65) circle (0.05);
\node at (4.5,-1) {$\scriptstyle B$};
\end{tikzpicture}
\caption{Schematic illustration of field concentration generator $B$ near  $\Omega$.}
\label{omega}
\end{figure}
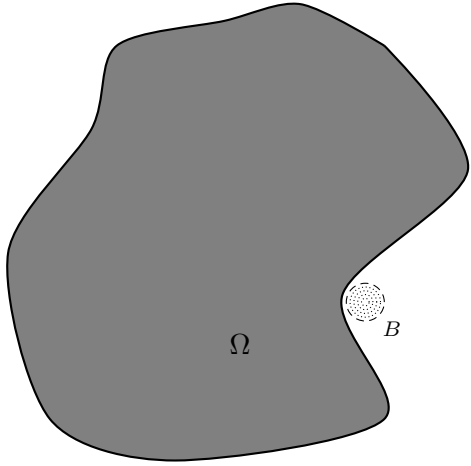

To enhance the clarity of the idea of generating field concentration, we outline the main proof of Theorem \ref{11} in this paper.
Firstly, we establish the existence of a sequence of specific transmission eigenfunctions within a radial domain for an any given wavenumber (refer to Lemma \ref{31}). These transmission eigenfunctions are dependent on the construction of material parameters and as surface waves, they exhibit strong surface resonant properties (refer to Theorem \ref{oscillating}). These properties satisfy the requirements for generating field concentration.
Secondly, the transmission eigenvalue problem \eqref{spectral}  possesses a distinctive spectral pattern as a result of the topological structure of the defined domain (refer to Proposition \ref{transrelation}). Based on this spectral pattern, we can construct the corresponding transmission eigenfunctions.
Thirdly, there exists a Herglotz wave function that can serve as a desirable approximation to the transmission eigenfunction. This Herglotz wave function satisfies the entire Helmholtz equation \eqref{incident} and can be assigned as the incident field (refer to Lemma \ref{appr}).
Finally, by utilizing the nearly vanishing scattering property in the presence of transmission eigenfunctions (refer to Lemma \ref{scattersmall}), we can rigorously prove that the gradient of the total field would blow up at specified locations around the inhomogeneous medium inclusion.

As aforementioned in the abstract, our study critically relies on certain geometric properties of the so-called transmission eigenfunctions, which were recently revealed in \cite{Chow2023, Chow2021}. The transmission eigenvalue problem arises in the study of invisibility/transparency in continuum mechanics and has a very colorful history. For further discussion on this topic, we refer the readers to \cite{Liu2022}.
In this paper, we reexamine the transmission eigenvalue problem \eqref{spectral} from the standpoint of material construction.
We have proved that for an any given wavenumber $k$, there exists a sequence of transmission eigenfunctions which exhibit peculiar surface resonant behaviors, accompanying strong surface-localization and surface-oscillation patterns along with
 the energy blowup. 
The main result is presented as follows.

\begin{Theorem}\label{oscillating}
Consider the transmission eigenvalue problem \eqref{spectral} in a radial domain $B_{r_0}$ centering the origin with radius $r_0$ in $\bbr^N, N=2,3$ and a given constant wavenumber $k$. Then there exists a sequence of transmission eigenfucntions $(v_m,w_m)_{m \in \bbn}$  related to the material parameters $(\sigma_m,\tau_m)_{m \in \bbn}$ such that $k$  is the transmission eigenvalue and the transmission eigenfunctions themselves and their gradients are surface-localized, i.e.
\begin{equation*}
 \lim_{m\rightarrow \infty}\frac{\| \psi_m\|_{L^2(B_{\xi )}}}{\| \psi_m\|_{L^2(B_{r_0})}} =0 \coma
 \lim_{m\rightarrow \infty}\frac{\| \nabla \psi_m\|_{L^2(B_{\xi})}}{\|\nabla \psi_m\|_{L^2(B_{r_0})}} = 0  , \quad \psi_m =w_m, v_m ,
\end{equation*}
where  $B_{\xi}$ is defined as
\begin{equation*}
B_{\xi }= \{x \in B_{r_0}, \mathrm{dist}(x,\partial B_{r_0})\geqslant \xi r_0,\xi \in (0,1)\}.
\end{equation*}
We further suppose $v_m$ is $L^2(B_{r_0})$-normalized, i.e. $\|v_m\|_{L^2(B_{r_0})} =1$, then
\begin{equation*}
\lim_{m\rightarrow \infty} \frac{\|\nabla \psi_m\|_{L^\infty(\Sigma_{\xi )}}}{k} = \infty ,
\end{equation*}
where $\Sigma_{\xi }$ is defined as
\begin{equation*}
\Sigma_{\xi} =
\begin{cases}
\{ (r,\theta) | \xi r_0 < r < r_0, \theta_1 < \theta <\theta_2\} \coma \xi \in (0,1); \theta_1,\theta_2 \in  [0,2\pi) , \\
 \{(r,\theta,\varphi) | \xi r_0 < r < r_0,  0\leqslant \theta \leqslant \pi ,\varphi_1 < \varphi <\varphi_2\} \coma \xi \in (0,1); \varphi_1,\varphi_2 \in [0,2\pi).
\end{cases}
\end{equation*}
\end{Theorem}
\begin{Remark}
$B_{\xi }$ and $\Sigma_{\xi}$ represent the interior and boundary domain inside the inclusion respectively. In theorem \ref{oscillating},
the first limit indicates that these transmission eigenfucntions are surface-localized eigenmodes.
The second limit indicates that their oscillating behaviors tend to concentrate along the boundary for both transmission eigenfunctions.
In the vicinity along the boundary, The last limit indicates that the oscillating frequencies of transmission eigenfunctions are much larger than the natural frequency $k$, especially for sufficiently large order $m$.
\end{Remark}

\subsection{Background discussion and scientific novelties}

In the field of fluid mechanics, a field concentration refers to a location within an object where the gradient of velocity potential or pressure is substantially higher than the surrounding region in the fluid field.
Excessively large velocity gradient can lead to serious consequences, including flow instability and structural damage.
Nevertheless, the destructive nature of field concentration can be harnessed for practical applications.
Proper utilization of field concentration offers engineers numerous possibilities to analyze and optimize the performance of fluid systems.
Therefore, comprehending field concentrations is fundamental in fluid mechanics and plays a vital role in designing and optimizing various engineering applications.

Field concentration can generally be utilized to perform destructive tasks in engineering.
In medical applications, such as treating cancers (e.g., prostate cancer, liver tumors) and non-cancerous conditions (e.g., uterine fibroids), a transducer focuses acoustic waves on specific points within the body to create high-intensity regions where the acoustic energy concentrates and lead to thermal ablation of the target tissue. The concentrated acoustic field can generate localized energy concentration capable of destroying pathological tissues, including tumors.
Similarly, in the treatment of kidney stones and gallstones, lithotripsy involves generating a series of acoustic waves outside the body, concentrating acoustic energy at the location of the stone and render it fragile. Subsequently, field concentration can fragment kidney stones or gallstones into smaller pieces that can be naturally expelled by the body.
On the other hand, field concentration can also enhance the performance of fluid systems.
In the design of aircraft lifting surfaces, the low-pressure region on the upper surface of the airfoil and the corresponding pressure gradients are utilized to generate the required lift. Optimizing the distribution of field concentration can improve lift and reduce drag.
Nozzles and propulsion systems also exemplify this application. The design of rocket engines and jet engines relies on the extremely high velocity potential gradient and pressure gradient in the exhaust flow to generate thrust. Through meticulous nozzle system design, the velocity potential gradients can be optimized to enhance the propulsive efficiency.

As mentioned earlier, we are concerned with artificially creating strong field concentrations in a customized way.  The field concentration can occur at a customized place around a generic material inclusion or within the material inclusion if it possesses a certain specific geometric structure.
There are several novel and salient features of our study. First, the field concentration has been widely investigated in the literature, say e.g. \cite{Yun2009,Ammari2007,Dong2021,Bao2009} in the high-contrast composite material theory. It turns out that the material and geometric irregularities are crucial for the occurrence of field concentration. In fact, it is usually required that the material parameters of the inclusion  and those of background medium possess a certain asymptotically high contrast. Moreover, it usually occurs between two close-to-touching convex material inclusions, corresponding to the building blocks of the underlying composite material. The study is usually concerned with the static or quasi-static case, namely $\omega = 0$ or $w \cdot \mathrm{diam}(\Omega) \ll 1$ see \cite{Deng2022,Deng2022a,Ammari2020} for discussion on the related developments in the literature. However, for our study, there are none of such restrictive requirements: there can be only one material inclusion with regular material parameters of generic geometric shape, and the field concentration can occur in the quasi-static regime or beyond the quasi-static regime. Second, a large field concentration may cause the failure of the material structure. Many preventive methods have been developed to counter the destructive effect of field concentration. To our best knowledge, there is little study on artificially and customarily generating strong field concentration within a material structure. As is well known, everything might have many different sides. The destructive nature of field concentration can be put into good use. For example, in the treatment for kidney stones and gallstones, a destructive failure of the  stone structure inside can make the treatment easier and safer and reduce the need for surgical intervention. Finally, we would like to mention that gradient estimates of solutions is a central topic in the theory of partial differential equations (PDEs); see e.g. \cite{li2000gradient}. Our study provides a good example that for wave equations, the gradient blowup phenomena may be even induced for (possibly) smooth coefficients due to frequency effect.

The rest of the paper is organized as follows.
Section 2 presents the spectral and geometric patterns of surface transmission resonance and provides the proof for the existence of a sequence of surface transmission resonant eigenfunctions associated with the given wavenumber. These eigenfunctions are utilized to construct transmission resonant balls, which facilitate the generation of strong field concentration.
Section 3 demonstrates the artificial generation of customized field concentration via the spectral and geometric patterns of transmission resonance. It is followed by the illustration of this phenomenon through numerical experiments in Section 4.
Section 5 focuses on the surface-oscillating behaviors of transmission eigenfunctions and concludes the complete proof of Theorem \ref{oscillating}.

\section{Surface transmission resonance}\label{section2}
This section introduces the transmission eigenvalue problem along with its spectral patterns, and subsequently establishes the existence of a sequence of surface transmission resonant eigenfunctions.

\subsection{Transmission eigenvalue problem and its spectral patterns}
The transmission eigenvalue problem is associated with  invisibility in  the wave scattering system \eqref{sys1}. When invisibility occurs, there is no perturbation outside caused by the inhomogeneous medium inclusion,  i.e., $u^s\equiv 0$. This reduction leads to the following (interior) transmission eigenvalue problem:
\begin{equation}\label{spectral}
\begin{cases}
\displaystyle \nabla \cdot \left(\sigma\nabla w \right)  +k^2\tau w = 0 &\text{in} \; B, \medskip \\
\displaystyle \Delta  v +k^2v = 0 &\text{in} \  B, \medskip\\
\displaystyle w =v \coma \sigma\frac{\partial w}{\partial \nu} = \frac{\partial v}{\partial\nu} &\text{on} \ \partial B,
\end{cases}
\end{equation}
where $\nu \in \mathcal{S}^{N-1}$ is the exterior unit normal to $\partial B$ and $B$ is a bounded Lipschitz domain in $\bbr^N$.
In order to make a distinction between the inhomogeneous medium inclusion introduced in \eqref{sys1} and the definition domain  introduced in \eqref{spectral}, we introduce the symbol $B$ to denote the definition domain for transmission eigenvalue problem. By slight abuse of notation, $B_r(\bx_0)$ represents a radial domain with radius $r$ and center $\bx_0$ in $\bbr^N$  that should be evident from the context.. The center $\bx_0$ is omitted when it is the origin.
Additionally, we introduce an auxiliary parameter to facilitate our analysis:
\begin{equation}\label{nn}
n =\sqrt{\frac{\tau}{\sigma}}.
\end{equation}
It is evident that  $w = v = 0$ forms  a pair of trivial solutions to \eqref{spectral}.
If the system \eqref{spectral} admits a pair of nontrivial solutions $(v,w)$, then $k$ is referred to as a transmission eigenvalue, and $(v,w)$ are the corresponding transmission eigenfunctions. The transmission eigenfunction in fact can be viewed as the restriction of the wave inside the inhomogeneous medium inclusion. The behaviors of transmission eigenfunction reflect the geometric properties of  the wave probing on the invisible inclusion. This can help to understand the mathematical nature of invisibility.
In this paper, we find out there exists a sequence of surface-localized transmission eigenfunctions $(v_m,w_m)_{m \in \bbn}$. Here, the localization  means that there exists a sufficiently small number $\varepsilon =o(1)$ such that
\begin{equation*}
\frac{\|\psi\|_{L^2(B_\varepsilon)}}{\|\psi\|_{L^2(B)}} = o(1) , \quad  \psi = w \; \textrm{and} \; v.
\end{equation*}
where $B_\varepsilon:= \{x \in B, \textrm{dist}(x,\partial B)\geqslant \varepsilon\}.$
If localization occurs, the corresponding transmission eigenfunctions are called surface-localized transmission eigenmodes.

Next we shall consider the spectral pattern of transmission eigenvalue problem.  Let $\Omega$ be composed of multiple simply-connected components as follows:
\begin{equation*}
(\Omega;\sigma,\tau) = \mathop{\cup}_{j=1}^{L_0}(\Omega_j;\sigma_j,\tau_j),
\end{equation*}
where $\Omega_j$'s are mutually disjoint and $(\Omega_j;\sigma_j,\tau_j)$ are the restriction of $(\Omega;\sigma,\tau)$ on $\Omega_j$.
The following theorem states that the set of all transmission eigenvalues in $\Omega$ is in fact the union of the eigenvalues in each component $\Omega_j$. Moreover, the transmission eigenfunctions in $\Omega$ are linear combinations of the eigenfunctions in each component $\Omega_j$ with the trivial eigenfunctions in other components.

\begin{Proposition}\label{transrelation}
Let $\boldsymbol{\sigma} [\Omega;\sigma,\tau]$ and $\boldsymbol{\sigma}[\Omega_j;\sigma_j,\tau_j] $ be the set of transmission eigenvalues of the system \eqref{spectral} in $(\Omega;\sigma,\tau)$ and  $(\Omega_j;\sigma_j,\tau_j), 1\leqslant j \leqslant L_0$,  where $(\Omega;\sigma,\tau) = \mathop{\cup}_{j=1}^{L_0}(\Omega_j;\sigma_j,\tau_j)$. Then
\begin{equation} \label{thm}
\boldsymbol{\sigma} [\Omega;\sigma,\tau] = \mathop{\cup}\limits_{j=1}^{L_0}\boldsymbol{\sigma}[\Omega_j;\sigma_j,\tau_j].
\end{equation}
\end{Proposition}
\begin{proof}
We first prove that
\begin{equation}\label{eq:p1}
\mathop{\cup}_{j=1}^{L_0} \boldsymbol{\sigma}[\Omega_j;\sigma_j,\tau_j]\subset \boldsymbol{\sigma} [\Omega;\sigma,\tau].
\end{equation}
Without loss of generality, we consider $k\in  \boldsymbol{\sigma}[\Omega_1;\sigma_1,\tau_1]$ with the associated transmission eigenfunctions $w_1\in H^1(\Omega_1)$ and $v_1\in H^1(\Omega_1)$, that is,
\begin{equation}\label{eq:s3}
\begin{cases}
\nabla \cdot (\sigma_1\nabla w_1) +k^2 \tau_1 w_1 = 0 &\text{in} \quad \Omega_1, \\
\Delta v_1 +k^2v_1 = 0 &\text{in} \quad \Omega_1, \\
w_1 =v_1 \coma \sigma_1\frac{\partial w_1}{\partial \nu} = \frac{\partial v_1}{\partial\nu} &\text{on} \quad\partial \Omega_1,
\end{cases}
\end{equation}
Next, it is trivially seen that $w_j=v_j=0$, $j=2, 3,\ldots, L_0$, satisfy
\begin{equation}\label{eq:s4}
\begin{cases}
\nabla \cdot (\sigma_j\nabla w_j) +k^2 \tau_j w_j = 0  &\text{in} \quad \Omega_j, \\
\Delta v_j +k^2v_j = 0 &\text{in} \quad \Omega_j, \\
w_j =v_j \coma \sigma_j\frac{\partial w_j}{\partial \nu} = \frac{\partial v_j}{\partial\nu} &\text{on} \quad\partial \Omega_j,
\end{cases}
\end{equation}
Set
\begin{equation}\label{eq:p2}
\begin{split}
w=& w_1\chi_{\Omega_1}+0\cdot\chi_{\Omega_2}+\cdots+0\cdot\chi_{\Omega_{L_0}}\in H^1(\Omega),\\
v=& v_1\chi_{\Omega_1}+0\cdot\chi_{\Omega_2}+\cdots+0\cdot\chi_{\Omega_{L_0}}\in H^1(\Omega).
\end{split}
\end{equation}
By virtue of \eqref{eq:s3} and \eqref{eq:s4} as well as using the fact that $\Omega_j$ are mutually disjoint, it is directly seen that $w$ and $v$ defined in \eqref{eq:p2} are nontrivial solutions to \eqref{spectral} with $k=k_1$. Hence, $k_1\in \boldsymbol{\sigma} [\Omega;\sigma,\tau]$, which readily proves \eqref{eq:p1}.

We proceed to prove that
\begin{equation}\label{eq:p3}
\boldsymbol{\sigma} [\Omega;\sigma,\tau]\subset \mathop{\cup}_{j=1}^{L_0} \boldsymbol{\sigma} [\Omega_j; \sigma_j,\tau_j].
\end{equation}
Let $k\in \boldsymbol{\sigma} [\Omega;\sigma,\tau]$ and $w, v$ be the associated transmission eigenfunctions. Since $w\in H^1(\Omega)$ and $v\in H^1(\Omega)$ are non-identically zero and $\Omega_j$'s are mutually disjoint, there must exist an $\Omega_{j_0}$, $1\leq j_0\leq L_0$, such that
\begin{equation*}
w_{j_0}=w|_{\Omega_{j_0}}\in H^1(\Omega_{j_0}), \quad v_{j_0}=v|_{\Omega_{j_0}}\in H^1(\Omega_{j_0}),
\end{equation*}
are non-identically zero. Again, by using the facts that $\Omega_j$'s are mutually disjoint and $w, v$ satisfy \eqref{spectral}, it can be seen that
\begin{equation*}
\begin{cases}
\nabla \cdot (\sigma_{j_0}\nabla w_{j_0}) +k^2 \tau_{j_0} w_{j_0} = 0  &\text{in} \quad \Omega_{j_0}, \\
\Delta v_{j_0} +k^2v_{j_0} = 0 &\text{in} \quad \Omega_{j_0}, \\
w_{j_0} =v_{j_0} \coma \sigma_{j_0}\frac{\partial  w_{j_0}}{\partial \nu} = \frac{\partial v_{j_0}}{\partial\nu} &\text{on} \quad\partial \Omega_{j_0},
\end{cases}
\end{equation*}
That is,
\begin{equation*}
k\in \boldsymbol{\sigma}[\Omega_{j_0};\sigma_{j_0},\tau_{j_0}]\subset \mathop{\cup}_{j=1}^{L_0} \boldsymbol{\sigma} [\Omega_j;\sigma_{j},\tau_j],
\end{equation*}
which proves \eqref{eq:p3}.

Finally, combining \eqref{eq:p1} and \eqref{eq:p3} readily yields (\ref{thm}), thus completing the proof.
\end{proof}

\subsection{The geometric patterns of transmission eigenfunctions}
The transmission eigenvalue problem \eqref{spectral} possesses a sequence of transmission eigenvalues  and surface-localized transmission eigenfunctions for the given material parameters $(\sigma,\tau)$ \cite{Deng2021, Jiang2022}. In this subsection, we reexamine the transmission eigenvalue problem for a given wavenumber $k$ and  radial domain $B_{r_0}$. We prove that there exists a sequence of $\{n_m\}_{m \in \bbn}$   such that $k$ is the transmission eigenvalue and the corresponding  transmission eigenfunctions $(v_m,w_m)_{m \in \bbn}$ are surface-localized. This allows us  to construct the field concentration generator with arbitrary size from the standpoint of material parameters.
In fact, these eigenfunctions are  surface transmission resonant modes, accompanying strong surface-localized and surface-oscillating patterns along with the energy blowup.
Since the verification of surface-oscillating patterns involves technical and tedious calculations, we defer the related proof to the section \ref{oscillation}.

To begin with, we denote $J_m(x)$ and $j_m(x)$  respectively denote the $m$-th order Bessel function and $m$-th order spherical Bessel function. Moreover, let $j_{m,s}$ and $j'_{m,s}$ respectively denote the $s$-th positive root of $J_m(x)$ and $J'_m(x)$ (arranged according to magnitude).

\begin{Lemma}\label{31}
Consider the transmission eigenvalue problem \eqref{spectral} and assume $B_{r_0}$ is a radial object with radius $r_0$ in $\bbr^N$ centering the origin. Given any wavenumber $k>0$, there exists a sequence of $\{n_m\}_{m\in \bbn}$ defined in \eqref{nn} such that $k$ is the transmission eigenvalue and the corresponding transmission eigenfunctions $(v_m,w_m)$ are both surface-localized. In specific, Given  $s_0 \in \bbn$ be a fixed positive integer,  the material parameter $n_m$ in $\bbr^2$ satisfies
\begin{equation}\label{eigenvalue2}
n_m \in (\frac{j_{m,s_0}}{kr_0},\frac{j_{m,s_0+1}}{kr_0}) ,
\end{equation}
and the corresponding eigenfunctions are
\begin{equation}\label{eigenfunction2}
 v_m(\bx) = \beta_m J_m(kr)e^{\mi m\theta} \coma w_m(\bx) = \alpha_mJ_m(kn_mr)e^{\mi m\theta} ,
\end{equation}
where  $\bx = (r\cos \theta,r\sin \theta) \in \bbr^2$ denotes the polar coordinate. 
In $\bbr^3$, $n_m$ satisfies
\begin{equation}\label{eigenvalue3}
 n_m \in (\frac{j_{m+\frac{1}{2},s_0}}{kr_0},\frac{j_{m+\frac{1}{2},s_0+1}}{kr_0})
\end{equation}
The corresponding eigenfunctions in $\bbr^3$ are
\begin{equation}\label{eigenfunction3}
 v_m^l(\bx) = \beta_m^l j_m(kr)Y_m^l(\theta,\varphi) \coma w_m^l(\bx) = \alpha_m^l j_m(kn_mr)Y_m^l(\theta,\varphi) \coma  -m \leqslant  l \leqslant m ,
\end{equation}
where $\bx = (r\sin \theta \cos \varphi,r \sin\theta \sin\varphi,r\cos \theta) \in \bbr^3$ denotes the spherical coordinate and $Y_m^l$ are the spherical harmonics.

\end{Lemma}
\begin{proof}
To ensure that $(v_m,w_m)$ are the transmission eigenfunctions of \eqref{spectral}, the transmission boundary conditions on the last line in \eqref{spectral} yield
\begin{equation*}
\alpha_m J_m(knr_0) = \beta_m J_m(kr_0) \quad \text{and} \quad \sigma n \alpha_m  J'_m(knr_0) = \beta_m J'_m(kr_0).
\end{equation*}
So $n$ must be the root of the following equation:
\begin{equation*}
f_m(n) = J'_{m}(kr_0)J_m(knr_0) -\sigma nJ_m(kr_0)J'_{m}(knr_0).
\end{equation*}
For a given integral $s_0 \in \bbn$, it is obvious that $J_m(knr_0) = 0$ when $n = j_{m,s_0}/(kr_0)$ and $j_{m,s_0+1}/(kr_0)$ respectively.
From the fact that the zeros of $J_m(x)$ and $J'_{m}(x)$ are interlaced, it holds that
\begin{equation*}
f_m(\frac{j_{m,s_0}}{kr_0}) \cdot f_m(\frac{j_{m,s_0+1}}{kr_0}) = \sigma^2 \frac{j_{m,s_0}j_{m,s_0+1}}{k^2r_0^2} J^2_m(kr_0)J'_{m}(j_{m,s_0})J'_{m}(j_{m,s_0+1}) < 0.
\end{equation*}
It readily implies by Rolle's theorem that there exists at least one zero denoted by $n_m$ in $(j_{m,s_0}/(kr_0),\linebreak[2]j_{m,s_0+1}/(kr_0))$ satisfying $f_m(n_m) =0$. We have proven that $n_m$ satisfying \eqref{eigenvalue2} and $(v_m,w_m)$ are the corresponding transmission functions associated with the transmission eigenvalue $k$.

 The surface-localizations of $(v_m,w_m)$ mainly depend on the properties of Bessel functions. From the orthogonality of $\{e^{\mi m\theta}\}_{m \in \bbn}$ in the unit circle, we can obtain
\begin{equation}\label{vm}
\|v_m\|^2_{L^2(B_{r_0})} =\beta_m^2 \int_0^{r_0}  J^2_m(kr)r \dx{r} =  \frac{\beta_m^2}{k^2} \int_0^{kr_0} J^2_m(r)r \dx{r}.
\end{equation}
Without loss of generality, we assume $kr_0$ is always less than $m$ and  this assumption always holds when $m$ is sufficiently large. We define a monotonously increasing and convex function
\begin{equation*}
f(r) = J^2_m(r)r \coma r \in (0,kr_0).
\end{equation*}
In fact,  straight calculations yield
\begin{align*}
f'(r) &= J_m(r)(2J'_m(r)r+J_m(r)) >0 ,\\
f''(r) &= 2{J'^2_m}(r)r + 2{J'_m}(r)J_m(r) + 2r^{-1}(m^2-r^2)J_m^2(r) >0.
\end{align*}
The convex property means that the integral \eqref{vm} is bigger than the area of the triangle under the tangent of $f(kr_0)$ with the $x$-axis, namely,
\begin{equation*}
 \int_0^{kr_0} J^2_m(r)r \dx{r} \geqslant  \frac{1}{2}\frac{J_m^3(kr_0)(kr_0)^2}{J_m(kr_0)+2kr_0J'_m(kr_0)}.
\end{equation*}
Since $f(r)$ is  monotonously increasing, it is obvious that
\begin{equation*}
 \int_0^{k\xi r_0} J^2_m(r)r \dx{r} \leqslant \frac{1}{2} J^2_m(k\xi r_0)(k\xi r_0)^2 \coma \xi \in (0,1) .
\end{equation*}
Combining the above two estimates, we further obtain
\begin{equation}\label{limitsv}
\begin{aligned}
\frac{\|v_m\|_{L^2(B_{\xi})}^2}{\|v_m\|_{L^2(B_{r_0})}^2} &\leqslant \frac{J_m^2(k\xi r_0)(k\xi r_0)^2}{J_m^3(kr_0)(kr_0)^2/(J_m(kr_0)+2kr_0J'_m(kr_0))}  \\
&= \left(\frac{J_m(k\xi r_0)}{J_m(kr_0)}\right)^2 \xi^2\left(1 + 2kr_0\frac{J'_m(kr_0)}{J_m(kr_0)} \right).
\end{aligned}
\end{equation}
When $m$ is sufficiently large, it holds that
\begin{equation}\label{esti1}
\frac{J'_m(kr_0)}{J_m(kr_0)} = \frac{m}{kr_0} - \frac{J_{m+1}(kr_0)}{J_{m}(kr_0)} = \frac{m}{kr_0}\left(1 - \frac{k^2r_0^2}{2m(m+1)} + \mathcal{O}(m^{-3})\right).
\end{equation}
where we have used the asymptotic expansions of $J_m$ for larger order,
\begin{equation}\label{Besselfunction}
J_m(x) = \frac{x^m}{2^m \Gamma(m+1)}\left(1 - \frac{x^2}{4(m+1)} + \mathcal{O}(m^{-2}) \right),
\end{equation}
and $\Gamma(x)$ is the Gamma function. By using the asymptotic expansions \eqref{Besselfunction} again, we can get
\begin{equation}\label{esti2}
\frac{J_m(k\xi r_0)}{J_m(kr_0)} = \frac{(k\xi r_0)^m\left(1 - \frac{(k\xi r_0)^2}{4(m+1)} + \mathcal{O}(m^{-2}) \right)}{(kr_0)^m\left(1 - \frac{(kr_0)^2}{4(m+1)} + \mathcal{O}(m^{-2}) \right)} \sim \xi^m.
\end{equation}
Finally, we insert \eqref{esti1} and \eqref{esti2} in  \eqref{limitsv} and get
\begin{equation*}
\frac{\|v_m\|_{L^2(B_{\xi})}^2}{\|v_m\|_{L^2(B_{r_0})}^2} \leqslant C(k,r_0)\xi^{2m+2}(1+2m) \rightarrow 0  \quad \textrm{as} \quad  m \rightarrow \infty.
\end{equation*}
We have proved the surface-localization of $\{v_m\}_{m\in \bbn}$. Similarly,
the $L^2$-norm of $w_m$ is
\begin{equation*}
\|w_m\|^2_{L^2(B_{r_0})} =\alpha_m^2 \int_0^{r_0}  J^2_m(kn_mr)r \dx{r} =  \frac{\beta_m^2}{k^2n_m^2} \int_0^{kn_mr_0} J^2_m(r)r \dx{r}.
\end{equation*}
Since $k n_m r_0 \in (j_{m,s_0},j_{m,s_0+1})$ and the zeros of Bessel functions have the following asymptotic expansions,
\begin{equation}\label{216}
j_{m,s_0} = m + b_{s_0}m^{\frac{1}{3}} + \mathcal{O}(m^{-\frac{1}{3}}),
\end{equation}
where $b_{s_0}$ is some positive constant only depending on the fixed integer $s_0 \in\bbn$, then it holds that
\begin{equation*}
kn_mr_0 > j_{m,1} >j'_{m,1}>m > kn_m \xi r_0.
\end{equation*}
By following a similar argument as \eqref{limitsv}, we can readily get
\begin{equation}\label{wmestimate}
\begin{aligned}
\frac{\|w_m\|_{L^2(B_{\xi})}^2}{\|w_m\|_{L^2(B_{r_0})}^2} &=  \frac{\int_0^{kn_m\xi r_0} J^2_m(r)r \dx{r}}{\int_0^{kn_mr_0} J^2_m(r)r \dx{r}}\leqslant  \frac{\int_0^{kn_m\xi r_0} J^2_m(r)r \dx{r}}{\int_0^{m} J^2_m(r)r \dx{r}}\\
&\leqslant \frac{\frac{1}{2}J_m^2(kn_m\xi r_0)(kn_m\xi r_0)^2}{\frac{1}{2}J_m^3(m)m^2/(J_m(m)+2mJ'_m(m))} \\
&= \left(\frac{J_m(kn_m\xi r_0)}{J_m(m)}\right)^2 \left(\frac{kn_m\xi r_0}{m}\right)^2\left(1 + 2m\frac{J'_m(m)}{J_m(m)} \right).
\end{aligned}
\end{equation}
On the one hand, $kn_m\xi r_0 <m$ and
\begin{align*}
1 + 2m\frac{J'_m(m)}{J_m(m)} &= 1 +2m \left(1 - \frac{J_{m+1}(m)}{J_m(m)}\right) = 1 + m\left( \frac{m+2}{m+1} - \frac{m}{ (m+1)}\frac{J_{m+2}(m)}{J_m(m)}\right) \\
&\leqslant 1 +  \frac{m(m+2)}{m+1} \leqslant m+2.
\end{align*}
where we have used the recurrence formula $J_{m}(x) +J_{m+2}(x) = 2(m+1)J_{m+1}(x)/x$. On the other hand,  from \cite[(9.3.31),  (9.3.35)]{Abramowitz1988}, one can get
\begin{align*}
\frac{J_m(kn_m\xi r_0)}{J_m(m)} \sim C(k,\xi) Ai(m^{\frac{2}{3}} \zeta_{x}) \sim C(k,\xi)  e^{-\frac{2}{3}\zeta_x^{3/2}m}{m^{-1/6}},
\end{align*}
where
$$\zeta_x=\left(\frac{3}{2}\left(\ln \left(\frac{1+\sqrt{1-x^2}}{x}\right)-\sqrt{1-x^2}\right)\right)^{\frac{2}{3}}\coma x =k n_m\xi r_0/m \sim \xi ,$$
and we have used the  asymptotic formula of Airy function for sufficiently $x$,
\begin{equation*}
Ai(x) = \frac{e^{-\frac{2}{3}x^{3/2}}}{2\pi^{1/2}x^{1/4}}\left(1 - \frac{5}{48}x^{-\frac{3}{2}} +\mathcal{O}(x^{-3}) \right).
\end{equation*}
Finally, we can further obtain
\begin{align}\label{wmfinal}
\frac{\|w_m\|_{L^2(B_{\xi})}^2}{\|w_m\|_{L^2(B_{r_0})}^2} &\leqslant  C(k,\xi) e^{-\frac{2}{3}\zeta_x^{3/2}m}{(m+2)}{m^{-1/6}} \rightarrow 0 \quad \text{as} \ m \rightarrow \infty.
\end{align}
Therefore, we have proved the surface-localization of $\{w_m\}_{m\in \bbn}$.

The proof in $\bbr^3$ is similar with that in $\bbr^2$.
First, the transmission boundary conditions for $(v_m^l,w_m^l)$ defined on \eqref{eigenfunction3} yield
\begin{equation*}
f_m(n) = j'_{m}(kr_0)j_m(knr_0) -\sigma nj_m(kr_0)j'_{m}(knr_0).
\end{equation*}
It is noted that the zeros of $j_m(x)$ are the same as that of $J_{m+\frac{1}{2}}(x)$. Taking the value of $ n=\frac{j_{m+\frac{1}{2},s_0}}{kr_0}$ and $n =\frac{j_{m+\frac{1}{2},s_0+1}}{kr_0}$ respectively, we can get
\begin{equation*}
f_m(\frac{ j_{m+\frac{1}{2},s_0}}{kr_0})f_m(\frac{ j_{m+\frac{1}{2},s_0+1}}{kr_0}) = \sigma^2 \frac{ j_{m+\frac{1}{2},s_0}j_{m+\frac{1}{2},s_0+1}}{k^2r_0^2}j'_{m}(j_{m+\frac{1}{2},s_0})j'_{m}(j_{m+\frac{1}{2},s_0+1}).
\end{equation*}
It holds from $j_m(x)  =\sqrt{\frac{\pi}{2x}}J_{m+\frac{1}{2}}(x)$ that
\begin{align*}
\left. j'_{m}(x)\right|_{x = j_{m+\frac{1}{2},s_0}} = \left.\left( \sqrt{\frac{\pi }{2x}} J_{m+\frac{1}{2}}(x)\right)'\right|_{x = j_{m+\frac{1}{2},s_0}} = \sqrt{\frac{\pi}{2j_{m+\frac{1}{2},s_0}}}J'_{m+\frac{1}{2}}(j_{m+\frac{1}{2},s_0}).
\end{align*}
Since  the zeros between $J_{m+\frac{1}{2}}(x)$ and $J'_{m+\frac{1}{2}}(x)$ are interlaced, it is clear  that
\begin{equation*}
f_m(\frac{ j_{m+\frac{1}{2},s_0}}{kr_0})f_m(\frac{ j_{m+\frac{1}{2},s_0+1}}{kr_0})< 0.
\end{equation*}
By the  Rolle's theorem again, there exists at least one zero also denoted by $n_m$ satisfying \eqref{eigenvalue3} and $(v_m^l,w_m^l)$ are the corresponding transmission eigenfunctions.

The surface-localization of $(v_m,w_m)$ is obvious from the results in $\bbr^2$. Indeed, the $L^2$-norm of $(v_m,w_m)$ are given by
\begin{align*}
\|v^l_m\|^2_{L^2(B_{r_0})} &=(\alpha_m^l)^2 \int_0^{r_0}  j^2_m(kr)r^2 \dx{r} =  \frac{\pi(\alpha_m^l)^2}{2k^3} \int_0^{kr_0} J^2_{m+\frac{1}{2}}(r)r \dx{r}, \\
\|w^l_m\|^2_{L^2(B_{r_0})} &=(\beta_m^l)^2 \int_0^{r_0}  j^2_m(kn_mr)r^2 \dx{r} =  \frac{\pi(\beta_m^l)^2}{2k^2n_m^2} \int_0^{kn_mr_0} J^2_{m+\frac{1}{2}}(r)r \dx{r}.
\end{align*}
By substituting the order of Bessel function from $m$ to $m+1/2$ and following the same argument as \eqref{limitsv} and \eqref{wmestimate}, we can readily obtain the surface-localized results for  $(v_m^l,w_m^l)$.

\end{proof}

\section{Artificial generation of customized field concentrations}

In this section, we shall show how to artificially generate customized field concentration around the inclusion via surface transmission resonance.
To that end, we first introduce the Herglotz  approximation and prove the nearly vanishing property of scattered field in the presence of transmission eigenfunctions.



\subsection{Herglotz approximation}
A Herglotz function is a function of the form
\begin{equation}\label{def:her}
v_{g,k} = H_k{g} (\bx) := \int_{\mathbb{S}^{N-1}}e^{\mathrm{i}k\bx\cdot \theta} g(\theta) \dx{\theta} \coma \bx \in \bbr^N,
\end{equation}
where $g(\theta) \in L^2(\mathbb{S}^{N-1})$ is called the Herglotz kernel.
It is clear that the Herglotz function is an entire function  to the Helmholtz equation \eqref{incident} and can be designated as the incident field.
For the purpose of generating field concentration, we shall use a customized Herglotz function to activate the field concentration generator around the inhomogeneous medium inclusion.
Here we introduce the following denseness result of the Herglotz function.
\begin{Lemma} \cite{Weck2004}  \label{appr}
Suppose that $\Omega$ is of class $C^{\alpha,1} , \alpha \in \bbn \cup \{0\}$ with a connected complement in $\bbr^{N},N=2,3$. Denote by $\mathcal{H}$ be the space of all Herglotz functions and
\begin{equation*}
\mathfrak{H}(\Omega) := \{u \in C^\infty (\Omega) \coma \Delta u +k^2 u =0 \; in \; \Omega\}
\end{equation*}
Then $\mathcal{H}(\Omega)$ is dense in $\mathfrak{H}(\Omega)  \cap H^{\alpha +1}(\Omega)$ with respect to the $H^{\alpha+1}$ norm.
\end{Lemma}

\subsection{The nearly vanishing property of  scattered field}
We begin to prove the smallness of scattered field in the presence of transmission eigenfunction. This result is in fact a straight consequence of the
stability estimate for the the scattering system \eqref{sys1}.
To that end, we first truncate the scattering system \eqref{sys1} into a bounded ball $D_R$ with a sufficiently large radius $R$.
\begin{Lemma}\label{truncation}
The scattering system \eqref{sys1} is equivalent to the following truncated system:
\begin{equation}\label{truncated}
\begin{cases}
\Delta u^s_1 + k^2 u^s_1 = 0  & \text{in} \; D_R \backslash \overline{\Omega}, \medskip\\
\nabla \cdot(\sigma \nabla u_1) + k^2 \tau u_1 = 0  & \text{in} \;\Omega, \medskip\\
u_1 = u^s_1 + f \coma \sigma\partial_\nu u_1 = \partial_\nu u^s_1 + g &\text{on} \; \partial \Omega, \medskip\\
\frac{\partial u^s_1}{\partial \nu} = \Lambda u^s_1 & \textit{on} \; \partial D_R ,
\end{cases}
\end{equation}
where  $f = u^i|_{\partial \Omega} , g = \partial_\nu u^i|_{\partial \Omega}$ and
$\Lambda: H^{\frac{3}{2}}(\partial \Omega) \rightarrow H^{\frac{1}{2}}(\partial \Omega)$ is the Dirichlet-to-Neumann map defined by $\Lambda W := \frac{\partial V}{\partial \nu}$ with $V \in H_{loc}^2(\bbr^N \backslash \overline{D_R})$ is the unique solution of
\begin{equation}\label{radiating}
\begin{cases}
\Delta V + k^2 V = 0  & \text{in} \; \bbr^N \backslash \overline{D_R}, \\
V = W & \text{on} \;\partial D_R , \\
\displaystyle \lim\limits_{r \rightarrow \infty} r^\frac{N-1}{2}\left(\partial_r V-ik V\right) =  0 \coma &r= |\bx|.
\end{cases}
\end{equation}
\end{Lemma}
\begin{proof}
Let $(u,u^s) \in H^2(\Omega) \times H_{loc}^2(\bbr^N \backslash \overline{\Omega})$ be a solution of \eqref{sys1}. From the definition of $\Lambda$, the restriction of $u^s$ in $D_R \backslash \overline{\Omega}$ is clearly a solution of \eqref{truncated}. Conversely,  let  $(u_1,u^s_1) \in H^2(\Omega) \times H_{loc}^2(\bbr^N \backslash \overline{\Omega})$ be a solution of \eqref{truncated}. We shall prove that the extension of $u^s_1$ from $D_R \backslash \overline{\Omega}$ to outside of $D_R$ satisfies the Sommerfeld radiating condition. From the Green's formula for Helmholtz equation, the representation of $u^s_1$ is
\begin{align*}
u^s_1 =& \int_{\partial \Omega} u^s_1(y) \frac{\partial \Phi(x,y)}{\partial \nu(y)} - \frac{\partial u^s_1}{\partial \nu}(y) \Phi(x,y)\dx{s(y)} \\
&- \int_{\partial D_R} u^s_1(y) \frac{\partial \Phi(x,y)}{\partial \nu(y)} - \Lambda u^s_1(y) \: \Phi(x,y)\dx{s(y)}, \  x \in D_R \backslash  \overline{\Omega}, 
\end{align*}
where $\Phi(x,y)$ is the fundamental solution to Helmholtz equation in $\bbr^N$. From the definition of $\Lambda$, $u^s_1$ is in fact the boundary data of the radiating solution $V$ to \eqref{radiating}. Together with the fact that $\Phi$ is also a radiating solution to Helmholtz equation,  for $x \in D_R \backslash  \overline{\Omega}$, we can deduce that
\begin{align*}
&\int_{\partial D_R} u^s_1(y) \frac{\partial \Phi(x,y)}{\partial \nu(y)} - \Lambda u^s_1(y) \: \Phi(x,y)\dx{s(y)} \\
&= \lim_{r \rightarrow \infty}\int_{\partial B_r} V(y) \left(\frac{\partial \Phi(x,y)}{\partial \nu(y)} -\mi k \Phi(x,y) \right) - \left( \frac{\partial V}{\partial \nu}(y) - \mi k V(y)\right) \Phi(x,y)\dx{s(y)}  \\
& =0, 
\end{align*}
where we have used the boundedness of radiation solution and Sommerfeld radiation condition. Hence, $u_1^s$ coincides with the radiating solution to the \eqref{sys1} in the exterior of $\Omega$. We get $(u_1,u^s_1)$ is also a solution to \eqref{sys1}.
\end{proof}

\begin{Lemma}\label{scattersmall}
Consider the scattering system \eqref{sys1} with a $C^{1,1}$ inclusion $(\Omega;\sigma,\tau)$ in $\bbr^N, N=2,3$ and the wavenumber $k\in \bbr_+$.
We let $(v,w) $ be a pair of transmission eigenfunction in  $(\Omega;\sigma,\tau)$ associated with $k$.
If the incident field $u^i$ impinging on $(\Omega;\sigma,\tau)$ satisfies $\|u^i-v\|_{H^2(\Omega)} \leqslant \varepsilon$,  then there exists a constant $C(k,\sigma,\tau)$ such that the scattered field $u^s$ illuminated by $u^i$ has $\|u^s\|_{H^2_{loc}(\bbr^N \backslash \overline{\Omega})} \leqslant C \varepsilon$.
\end{Lemma}
\begin{proof}
For Lemma \ref{truncation}, we can rewrite \eqref{sys1} in the following form:
\begin{equation}\label{truncate2}
\begin{cases}
\Delta u^s + k^2 u^s = 0  & \text{in} \; D_R \backslash \overline{\Omega}, \medskip\\
\nabla \cdot(\sigma \nabla u) + k^2 \tau u = 0  & \text{in} \;\Omega, \medskip\\
u = u^s + f \coma \sigma\partial_\nu u = \partial_\nu u^s + g &\text{in} \; \partial \Omega, \medskip\\
\frac{\partial u^s}{\partial \nu} = \Lambda u^s & \text{on} \; \partial D_R ,
\end{cases}
\end{equation}
where $f = u^i|_{\partial \Omega}$ and $g = \partial_\nu u^i|_{\partial \Omega}$ and $\Lambda$ is defined in \eqref{radiating}. We first introduce the solution $u_f$ being the unique solution of
\begin{equation*}
\begin{cases}
\Delta u_f + k^2 u_f = 0  & \text{in} \; D_R \backslash \overline{\Omega}, \\
u _f = f  &\text{on} \; \partial \Omega, \\
u _f = 0  &\text{on} \; \partial D_R.
\end{cases}
\end{equation*}
Without loss of generality, we assume that  $k^2$ is not the Dirichlet eigenvalue of $-\Delta$ in $D_R \backslash \overline{\Omega}$ through a proper choice $D_R$.

By introducing the test function $\psi \in H^1(D_R)$, the equivalent variational formulation of \eqref{truncate2} is to find $p \in  H^1(D_R)$ such that
\begin{equation}\label{variration}
\begin{aligned}
&\int _{\Omega} \sigma \nabla p  \nabla \overline{\psi} -k^2 \tau p \overline{\psi} \dx{x}+ \int _{D_R\backslash \overline{\Omega}}  \nabla p \nabla \overline{\psi} -k^2  p \overline{\psi}\dx{x}  - \int_{\partial D_R} \Lambda p \overline{\psi}\dx{s}\\
=& \int _{D_R\backslash \overline{\Omega}}  -\nabla u_f  \nabla \overline{\psi} + k^2  u_f \overline{\psi}\dx{x} + \int_{\partial D_R}\partial_{\nu}u_f  \overline{\psi}\dx{s} + \int_{\partial \Omega}(g - \partial_{\nu}u_f )\overline{\psi}\dx{s} \coma \forall \psi \in H^1(D_R).
\end{aligned}
\end{equation}
By using Green's first theorem, it is easy to check that $u:= p|_{\Omega}$ and $u^s := p|_{ D_R \backslash \overline{\Omega}} -u_f$  satisfies the scattering system \eqref{truncate2}. On the other hand, we can also get the variational formulation \eqref{variration} by multiplying the test function to the first two equations in \eqref{truncate2}.

In order to prove the regularity result of \eqref{truncate2}, we first introduce a bounded operator $\Lambda_0:H^{1/2}(\partial D_R) \rightarrow H^{-1/2}(\partial D_R)$ defined in the Theorem 5.20 \cite{cakoni2005qualitative} and satisfies
\begin{equation*}
- \int_{\partial D_R} \Lambda_0 \psi \overline{\psi} \dx{s} \geqslant C \|\psi\|^2_{H^{\frac{1}{2}}(\partial D_R)} ,
\end{equation*}
for some $C>0$. Further, the operator $\Lambda -\Lambda_0$ is a compact operator from $H^{1/2}(\partial D_R)$ to $H^{-1/2}(\partial D_R)$. Then  the variational formulation \eqref{variration} can be rewritten as
\begin{equation}\label{interal}
a_1(p, \psi) + a_2(p, \psi) + \langle -(\Lambda-\Lambda_0)p,\psi\rangle = \mathcal{F}(\psi) \coma \forall \psi \in H^1(D_R),
\end{equation}
where $\langle \cdot ,\cdot \rangle$ denotes the inner product in $L^2(\partial D_R)$,  and the bilinear forms $a_1$ and $a_2$ together with the linear bounded functional $\mathcal{F}$ are defined by
\begin{align*}
a_1(p, \psi)  &:= \int _{\Omega} \sigma \nabla p  \nabla \overline{\psi} +  k^2 p\overline{\psi} \dx{x} + \int _{D_R\backslash \overline{\Omega}}  \nabla p  \nabla \overline{\psi} + k^2 p\overline{\psi} \dx{x} - \int_{\partial D_R} \Lambda_0 p \overline{\psi} \dx{s}, \\
a_2(p,\psi) &:= -\int _{\Omega} k^2(1+\tau)p \overline{\psi} \dx{x} -\int _{\Omega} 2k^2 p \overline{\psi} \dx{x} ,
\end{align*}
and
\begin{align*}
\mathcal{F}(\psi):=\int _{D_R\backslash \overline{\Omega}}  -\nabla u_f  \nabla \overline{\psi} + k^2  u_f \overline{\psi}\dx{x} + \int_{\partial D_R}\partial_{\nu}u_f  \overline{\psi}\dx{s} + \int_{\partial \Omega}(g - \partial_{\nu}u_f )\overline{\psi}\dx{s}.
\end{align*}
From the assumption of $\sigma$ and boundedness of $\Lambda_0$, it is clear that
\begin{equation*}
a_1(p,\psi) \leqslant C \|p\|_{H^1(D_R)}  \|\psi\|_{H^1(D_R)}  \coma a_1(p,p) \geqslant C \|p\|^2_{H^1(D_R)} ,
\end{equation*}
where $C$ is some positive constant. The Lax-Milgram lemma  indicates there exists a unique linear bounded operator $\mathcal{A}$ with a bounded inverse such that
\begin{equation*}
a_1(p,\psi) = (\mathcal{A}p,\psi) \coma \forall\,  p,\psi \in H^1(D_R).
\end{equation*}
By the Riesz representation theorem, there also exists a bounded operator $B:H^1(D_R) \rightarrow H^1(D_R)$ such that
\begin{equation*}
a_2(p,\psi) = (\mathcal{B}p,\psi) \coma  \forall p,\psi \in H^1(D_R).
\end{equation*}
We claim that $\mathcal{B}$ is compact in $H^1(D_R)$. In fact, Let $\{\phi_j\}$ be a bounded sequence  in $H^1(D_R)$ and the boundedness implies there exists a subsequence denoted also by  $\{\phi_j\}$ satisfying $\phi_j \rightharpoonup \phi$ for some $\phi \in H^1(D_R)$. On the other hand, $\phi_j \rightarrow \phi$ in $L^2(D_R)$ due to the Rellich–Kondrachov theorem. From the definition of $\mathcal{B}$, we know that $\{B\phi_j\}$ is weakly convergent in $H^1(D_R)$ and $(\mathcal{B}(\phi_j -\phi),\psi) =a_2 ((\phi_j -\phi),\psi)$. Let $\psi =\mathcal{B}(\phi_j -\phi),\psi)$, it is clear that
\begin{align*}
\|\mathcal{B}(\phi_j -\phi)\|_{H^1(D_R)} \leqslant  4k^2 \max \{\|1+\tau\|_{L^\infty(\Omega)},2\}\|\phi_j-\phi\|_{L^2(D_R)} \rightarrow 0.
\end{align*}
It means that $\mathcal{B}$ is compact in $H^1(D_R)$. The compactness of $\mathcal{C}$ in  $H^1(D_R)$ in the trace sense  is given in \cite{liu2012singular} where  $\mathcal{C}$ is defined as
\begin{equation*}
\langle -(\Lambda-\Lambda_0)p,\psi\rangle := (\mathcal{C}p,\psi) \coma  \forall p,\psi \in H^1(D_R).
\end{equation*}

We claim that $\mathcal{A} + \mathcal{B}+\mathcal{C}$ is bijective with a bounded inverse in $H^1(D_R)$ if it is injective. In fact, since $\mathcal{A}^{-1}$ is bijective and bounded, then
$$\mathcal{A} + \mathcal{B}+\mathcal{C} = \mathcal{A}(I - (-\mathcal{A}^{-1}( \mathcal{B}+\mathcal{C})),$$
implies the bijective equivalence of $\mathcal{A} + \mathcal{B}+\mathcal{C}$ and $I - (-\mathcal{A}^{-1}( \mathcal{B}+\mathcal{C}))$. From the fact that $\mathcal{A}^{-1}$ is bounded and $ \mathcal{B}+\mathcal{C}$ is compact, we can get that $-\mathcal{A}^{-1}(\mathcal{B}+\mathcal{C})$ is compact. The Fredholm theory gives the bijection of  $\mathcal{A} + \mathcal{B}+\mathcal{C}$ with a bounded inverse if it is injective. Therefore, To show the existence and uniqueness it suffices to show that $\mathcal{A} + \mathcal{B}+\mathcal{C}$  is injective, i.e. the only solution of homogeneous integral equation \eqref{interal} or the equivalent homogeneous system \eqref{truncate2} is identically zero. If this is done, by the Lax-Milgram theory the integral equation \eqref{interal} can be inverted in $H^1(D_R)$ and the inverse operator is bounded. From this, it follows that $(u,u^s)$ depends on the boundary data $(f,g)$.

The uniqueness of \eqref{truncate2} is guaranteed by the Rellich theorem. In specific, when $f=g =0$, we can get
\begin{align*}
\Im \int_{\partial D_R}\frac{\partial u^s }{\partial \nu}\overline{u^s}
&=\Im \left( \int_{D_R \backslash \overline{\Omega}}\Delta u^s \overline{u^s} + |\nabla u^s|^2 + \int_{\partial \Omega} \frac{\partial u^s }{\partial \nu}\overline{u^s} \right)\\
&= \Im \int_{\partial \Omega} \sigma \frac{\partial u}{\partial \nu} \overline{u} = \Im \int_{\partial \Omega}  \sigma|\nabla u|^2 - \Im k^2\int_{\partial \Omega}\tau|u|^2 \leqslant 0.
\end{align*}
Then $u^s  = 0$ in $\bbr^N \backslash \overline{\Omega}$. The homogeneous boundary data now imply that $u = \partial_\nu u = 0$ on $\partial \Omega$. From the unique continuity principle \cite{hormander2007analysis} we get $u = 0$ in $\Omega$. Summarizing all the above analysis, we have the continuous dependence on the boundary data of the incident field to the direct scattering problem in $H^1(D_R)$.  The standard elliptic  regularity estimate can increase the regularity from $H^1(D_R)$ to  $H^2(D_R)$ if the boundary data $(g,f) \in H^{\frac{1}{2}}(\partial \Omega) \times H^{\frac{3}{2}}(\partial \Omega)$.
We observe that the Sobolev norm of $u^i-v$ also gives the boundary information due to trace theorem.  Hence one has that
 \begin{equation*}
\|u^i-v\|_{H^{\frac{3}{2}}(\partial \Omega)} \leqslant \varepsilon \quad  \textit{and} \quad \|\partial_{\nu} u^i- \partial_{\nu} v \|_{H^{\frac{1}{2}}(\partial \Omega)} \leqslant \varepsilon.
\end{equation*}
We let $(w_j, w^s_j)$ be the unique solution to the system \eqref{truncate2} with the boundary conditions   $ f := u^i-v $ and $g := \partial_{\nu} u^i- \partial_{\nu} v $ on $\partial \Omega$ and $(g,f) \in H^{\frac{1}{2}}(\partial \Omega) \times H^{\frac{3}{2}}(\partial \Omega)$.
Then it is clear that
\begin{equation*}
\|w_j^s\|_{H^2_{loc} (\bbr^N \backslash \overline{\Omega})} \leqslant  C(k,\sigma,\tau)\varepsilon.
\end{equation*}
If the transmission eigenfunction $v$ can be exactly  extended to $v^i$ defined in the whole space, Then the scattered field $w^s$ illuminated by $v^i$ is identically zero due to the zero boundary conditions in \eqref{truncate2}. By the linearity of \eqref{sys1}, we can further get  $u^s = w_j^s + w^s$ and
\begin{equation*}
\|u^s\|_{H^2_{loc} (\bbr^N \backslash \overline{\Omega})} \leqslant \|w_j^s\|_{H^2_{loc} (\bbr^N \backslash \overline{\Omega})}  + \|w^s\|_{H^2_{loc} (\bbr^N \backslash \overline{\Omega})}   \leqslant C\varepsilon.
\end{equation*}
We have finished the proof.
\end{proof}

\subsection{Proof of Theorem \ref{11}}

With all the preliminary work above, we are in a position to prove the main results. The main idea of this proof is based on the spectral and geometric patterns of transmission eigenfunctions introduced in section \ref{section2}.
We mainly consider the three-dimensional case, and the two-dimensional case can be proven similarly except for the transmission eigenfunction $v_m^l$ substituted by $v_m$ defined in \eqref{eigenfunction2}.
\begin{proof}
We first define an exterior domain around the  inhomogeneous medium inclusion:
$$\Gamma_e(\Omega,\epsilon) = \{x \in \bbr^3 | \; \mathrm{dist}(x,\Gamma)\leqslant \epsilon, x \notin \overline{\Omega}\},$$
where $\Gamma$ is any subset of $\partial \Omega$ and $\epsilon = o(1)$ is an any given distance parameter. We can artificially generate field concentration phenomenon in $\Gamma_e(\Omega,\epsilon)$ by properly choosing an incident field $u^i$, namely $\|\nabla u\|_{L^\infty(\Gamma_e(\Omega,\epsilon))} >M $ for any given large $M$. The construction of $M$ shall be more definitely specified in what follows. In order to make the proof clearer, we divide it to several steps:

Step 1: We first construct the surface-localized transmission eigenfunctions. Let $B_{r_0} \in \bbr^3$ be a radial ball with radius $r_0>0$ satisfying
$$\epsilon : =\textrm{dist}(B_{r_0},\Gamma) = o(1).$$
By choosing a  proper coordinate system, we assume  $B_{r_0}$ centers at the origin.
We consider the transmission eigenvalue problem \eqref{spectral} in $B_{r_0}$.
From Lemma \ref{31}, there exists a sequence of material parameters $(\sigma_m,\tau_m)_{m \in \bbn}$  such that $k$ is  the transmission eigenvalue.  For example, we can assume
\begin{equation}\label{nm}
\sigma_m(\bx) = 1 \coma \tau_m(\bx) = \sqrt{n_m} \coma \bx \in B_{r_0},
\end{equation}
where $n_m$ is given in \eqref{eigenvalue3}.
The corresponding transmission eigenfunction $v_m^l$ defined in \eqref{eigenfunction3} is  surface-localized and satisfies the homogeneous Helmholtz equation \eqref{incident} in  $B_{r_0}$.

We further consider the transmission eigenvalue problem in $(\Omega; \sigma,\tau) \cup (B_{r_0};\sigma_m,\tau_m)$ where $(\sigma,\tau)$ and $(\sigma_m,\tau_m)$ are defined in \eqref{parameter} and \eqref{nm} respectively. Lemma \ref{transrelation} shows that the transmission eigenvalue $k$ in $B_{r_0}$ is also a transmission eigenvalue in $\Omega \cup B_{r_0}$.  Moreover, the corresponding transmission eigenfunction $\mathbf{v}_m^l$ is given by
\begin{equation}\label{incid}
\mathbf{v}_m^l = v_m^l\cdot \chi_{_{B_{r_0}}} + 0\cdot\chi_{_{\Omega}}.
\end{equation}
It is clear that $\mathbf{v}_m^l$ as the transmission eigenfunction  is the solution of homogeneous Helmholtz equation \eqref{incident} in $\Omega\cup  B_{r_0}$.

Step 2: We next construct the incident field with the Herglotz approximation.
In order to identify the transmission eigenfunction uniquely, we assume the normalized condition $\|\mathbf{v}_m^l\|_{L^2(\Omega \cup B_{r_0})} =1$. Then
\begin{equation*}
\beta_m^l = \frac{1}{\sqrt{\int_0^{r_0} j^2_m(kr)r^2\dx{r}}} = = \sqrt{\frac{2}{\pi}}\frac{k^{\frac{3}{2}}}{\sqrt{\int_0^{kr_0} J^2_{m+\frac{1}{2}}(r)r\dx{r}}}.
\end{equation*}
Since $\Omega$ is a of class $C^{1,1}$ and $\Omega \cup B$ has a connected complement in $\bbr^N$, hence lemma \ref{appr} shows that there exists a Herglotz wave $v_{g,k}$ such that $v_{g,k}$  is a $H^2(\Omega)$ approximation to $\mathbf{v}_m^l$. In other words, given any $\varepsilon >0$, it holds that
\begin{equation}\label{final}
\|v_{g,k} - \mathbf{v}_m^l\|_{H^2(\Omega \cup B_{r_0})} \leqslant \varepsilon.
\end{equation}
The Herglotz function $v_{g,k}$ is in fact the desired incident field since $v_{g,k}$ is the entire solution of Helmholtz equation \eqref{incident} in the whole space. Moreover, the Herglotz function $v_{g,k}$ exhibits strong resonant behaviors in  $B_{r_0}$ as the surface transmission eigenmode and nearly vanishes inside the inclusion $\Omega$.
It is noted that the construction of $v_{g,k}$ is based on the spectral and geometric patterns of transmission resonance. However, the existence of $v_{g,k}$ does not depend on the existence of $(B_{r_0};\sigma_m,\tau_m)$.  In fact, $B_{r_0}$ is an auxiliary radial domain with the same material parameter as the homogeneous background space. We denote the incident field by $u^i := v_{g,k}$.

Step 3: In the following, we analyze the behaviors of the total field $u$  illuminated by $u^i$ in \eqref{sys1}. The inequality \eqref{final} also indicates
\begin{equation}\label{herglotz}
\|u^i - \mathbf{v}_m^l\|_{H^2(\Omega)} \leqslant \varepsilon.
\end{equation}
From lemma \ref{scattersmall}, it is clear that the scattered field $u^{s}$ illuminated by $u^i$ would be sufficiently small outside $\Omega$. With  the Sobolev embedding  $H^2 \hookrightarrow C^0$,   we can further get
\begin{equation}\label{small}
\|u^s\|_{\infty, G\backslash \overline{\Omega}} \leqslant C(k,\sigma,\tau) \varepsilon,
\end{equation}
where $G$ is an any bounded domain containing $\Omega \cup B_{r_0}$. Therefore, the inclusion $\Omega$ is like invisible obstacle such that the behaviors of the incident field is almost undisturbed in the presence of the scatterer $\Omega$.

Since the behaviors of total field  is almost described by the incident field outside the homogeneous inclusion $\Omega$, we can investigate the behavior of incident field $v_{g,k}$. First,
\begin{align*}
v_m^l(\bx) = \beta_m^l j_m(kr) Y_m^l(\theta,\psi) =\frac{k}{\sqrt{\int_0^{kr_0} J^2_{m+\frac{1}{2}}(r)r\dx{r}}}J_{m+\frac{1}{2}}(kr)r^{-\frac{1}{2}}Y_m^l(\theta,\varphi) \coma \bx \in B_{r_0}.
\end{align*}
Without loss of generality, we always assume $m$ is sufficiently large such that $kr_0$ is less than $j'_{m+1/2,1}$. Then $J_{m+1/2}(kr)r^{-1/2}$ is monotonously increasing for $r \in (0,r_0)$ and attains the maximum ar $r = r_0$. By straight calculation, it holds that
$$
\max_{r\in (0,r_0)} \frac{J_{m+\frac{1}{2}}(kr)kr^{-\frac{1}{2}}}{\sqrt{\int_0^{kr_0} J^2_{m+\frac{1}{2}}(r)r\dx{r}}} = \frac{J_{m+\frac{1}{2}}(kr_0)kr_0^{-\frac{1}{2}}}{\sqrt{\int_0^{kr_0} J^2_{m+\frac{1}{2}}(r)r\dx{r}}} =  \frac{\sqrt{2m+3}}{r_0^{\frac{3}{2}}} \left(1 + \mathcal{O}(m^{-2}) \right).
$$
and
\begin{equation*}
\sup_{(\theta ,\varphi)\in T } Y_m^l (\theta,\varphi) = \sqrt{\frac{2m+1}{4\pi}} \sup_{\theta \in [0,\pi]}\sqrt\frac{(m-|l|)!}{(m-|l|)!}P_m^{|l|}(\cos \theta) \geqslant \sqrt{\frac{1.11}{4\pi}\frac{m+1/2}{|l|+1}} \coma  -m \leqslant l \leqslant m,
\end{equation*}
where $T := [0,\pi] \times [\varphi_1,\varphi_2],\varphi_1,\varphi_2 \in [0,2\pi)$  we have used lemma \ref{legendre}. Combining the above estimates, we can finally obtain
\begin{equation}\label{sobe}
\sup_{\bx \in \overline{B}_{r_0}} v_m^l(\bx) = \sup_{\bx \in \partial{B_{r_0}}} v_m^l(\bx)  \geqslant \frac{1}{\sqrt{2\pi}r_0^{3/2}}\frac{m+1/2}{\sqrt{|l|+1}}\left(1 + \mathcal{O}(m^{-2}) \right)\geqslant \frac{1}{3r_0^{3/2}}\frac{m+1/2}{\sqrt{|l|+1}},
\end{equation}
for sufficiently large $m$.

Step 4: We finally can prove the gradient of total field would blow up inside $\Gamma_e(\Omega,\epsilon)$. From the formulas \eqref{incid}, \eqref{herglotz}, \eqref{sobe} and the Sobolev embedding  $H^2 \hookrightarrow C^0$ inside $B_{r_0}\cup \Omega$, it holds that
\begin{equation*}
\sup_{\bx \in \partial{B_{r_0}}} u^i(\bx)  \geqslant \frac{1}{3r_0^{3/2}}\frac{m+1/2}{\sqrt{|l|+1}} - \varepsilon \quad \text{and} \quad \sup_{\bx \in \partial{\Omega}} u^i(\bx) \leqslant \varepsilon.
\end{equation*}
The estimate \eqref{small} implies that the scattered field  would be sufficiently small if we designate the incident field $u^i$ to be the Herglotz approximation $v_{g,k}$ to \eqref{incid}. It means that the total field  outside the scatterer $\Omega$ equals to the incident field plus a sufficiently small term. Let
$$
\bx_0 = \mathop{\arg\sup}_{\bx \in \partial B_{r_0}}v_{g,k}(\bx).
$$
we can assume $\epsilon  :=\textrm{dist}(B_{r_0},\Gamma) = \textrm{dist}(\bx_0,\Gamma)$ by rigid motions if necessary. Then
\begin{equation*}
 u(\bx_0) \geqslant \frac{1}{3r_0^{3/2}}\frac{m+1/2}{\sqrt{|l|+1}} -(1+C(k,\sigma,\tau)) \varepsilon \quad \text{and} \quad \sup_{\bx \in \partial{\Omega}} u(\bx) \leqslant (1+C(k,\sigma,\tau))\varepsilon.
\end{equation*}
Since $(1+C(k,\sigma,\tau))\varepsilon$ is sufficiently small, then we can assume
\begin{equation*}
(1+C(k,\sigma,\tau))\varepsilon \leqslant \frac{1}{12r_0^{3/2}}\frac{m+1/2}{\sqrt{|l|+1}}.
\end{equation*}
Finally by using the mean value theorem, we can obtain the conclusion
\begin{equation}\label{conclusions}
\|\nabla u \|_{L^\infty(\Gamma_e(\Omega,\epsilon))} \geqslant \frac{1}{\epsilon}\frac{1}{6r_0^{3/2}}\frac{m+1/2}{\sqrt{|l|+1}} \geqslant  \frac{1}{6}r_0^{-3/2} (m+\frac{1}{2})^{\frac{1}{2}}\frac{1}{\epsilon}.
\end{equation}
It is obvious from \eqref{conclusions} that $\nabla u$ would blow up when $m$ goes to infinity or $\epsilon$ goes to zero. Therefore,  For any given large number $M$, we can always construct the parameters $\epsilon,m$ and $r_0$ such that
\begin{equation*}
\frac{1}{6}r_0^{-3/2} (m+\frac{1}{2})^{\frac{1}{2}}\frac{1}{\epsilon} \geqslant M.
\end{equation*}
Here, the construction of parameters $\epsilon,m$ and $r_0$ depends on the  incident field $u^i$.
\end{proof}

\begin{Remark}
It is noted that the field concentration can occur at multiple localizations near the inclusion $\Omega$. In specific, we  can add multiple disjoint virtual field concentration generators $B_1 ,\cdots,B_{L_0}$ around the inclusion $\Omega$ for  $L_0 \in \bbn$. The incident field $u^i$ is designated to  Herglotz approximation to $\mathbf{v}_m^l = v_m^l\cdot \chi_{_{B_1}} + v_m^l\cdot \chi_{_{B_2}} + \cdots + v_m^l\cdot \chi_{_{B_{L_0}}} +  0\cdot\chi_{_{\Omega}}$ where $v_m^l$ are surface-localized transmission eigenfunctions up to a translation transformation. Thus,  the field concentration can occur in the intervals between $\Omega$ and $B_1 ,\cdots,B_{L_0}$ by following a similar argument.
\end{Remark}
\begin{Remark}
From the theorem \ref{11}, it is obvious that we can achieve field concentration in two customized ways. One is high-order probing. We can choose a high-order field concentration generator with  more intensive surface-localized property such that $m$ is sufficiently large. The other is in a geometric manner, i.e. the resonator is sufficiently close to the inclusion $(\epsilon \rightarrow 0)$.
\end{Remark}

\section{Numerical experiments}
In this section,  several numerical experiments  are presented to verify the theoretical results.  On the one hand, we shall determine the Herglotz wave, which is an approximation to transmission eigenfunction localizing on the surface of $B$ while almost vanishing within the interior of $\Omega$. On the other hand, we utilize the Herglotz wave as the incident field and obtain the total field by solving the scattering problem of the Helmholtz equation.

Given   a radial area $B$ near the scatterer $\Omega$,   we first obtain kernel function $g(x)$ by  solving the integral equation \eqref{def:her}. Noting that $H_k$ is compact, that is, the inverse of integral equation \eqref{def:her} is generally ill-posed,  then we use the Tikhonov regularization method as introduced in \cite{zhang2013herglotz}. We consider the integral equation
$   H_k g = f \coma   f  \in L^2 (\partial \Omega) \times L^2(\partial B)
$. The Tikhonov regularization is adopted to the following equation:
\begin{equation*}
    \alpha g_\alpha +H^*_gH_g g_\alpha = H^*_g f.
\end{equation*}
Thus, the regularized kernel function $g_\alpha$ is given by $ (\alpha + H^*_gH_g)^{-1}H^*_g f$. Next, we extend the  incident field to the whole domain by using the Herglotz approximation, namely, $u^i=H_kg_\alpha$.
Finally, we designate $u^i$ as the incident field and solve the inhomogeneous scattering system \eqref{sys1}.

In what follows, we present three examples to illustrate the phenomenon of field concentration.
In the first example, we consider an ellipse domain $\Omega$, i.e.,
\begin{equation*}
  \Omega:=\left\{z=(x,y)\; |\; \frac{(x+3.2)^2}{3^2 }+\frac{ (y+3.2)^2}{4^2} < 1\right\}.
\end{equation*}
 Here, we set the wavenumber $k=3$ and the parameter  $\sigma=1$.  By solving the transmission eigenvalue problem \eqref{spectral},  we obtain  $\tau=17.5285$, where the regularization parameter is chosen as $\alpha=10^{-5}$.
 The second example is an rectangle with width $a=4$ and height $b=8$, that is,
\begin{equation*}
 \Omega:=\{z=(x,y)\ |\  -5.1 <  x < 1.1 , -4 < y < 4 \},
\end{equation*}
where the parameters are set to be $k=1$ and $\sigma=1/4$. Furthermore,  we  have $\tau=16.2183$,  where the regularization parameter is chosen as $\alpha=10^{-6}$.
In the final example, we consider a kite domain $\Omega$, i.e.,
\begin{equation*}
  \Omega:=\left\{z=(x,y)\; |\; (x,y)=(\cos t+0.65\cos 2t+1.4, \ 1.5 \sin t), \ t\in[0,\, 2\pi]\right\}.
\end{equation*}
 Here, the parameters are set to be $k=\sqrt{3}$ and $\sigma=1/3$,  thus we get $\tau=19.9121$  and the regularization parameter is chosen as $\alpha=10^{-3}$.

		\begin{figure}[!ht]
		\centering
		\begin{minipage}{1\textwidth}
			\hfill\subfigure[$u$]{\includegraphics[width=0.33\textwidth]
                   {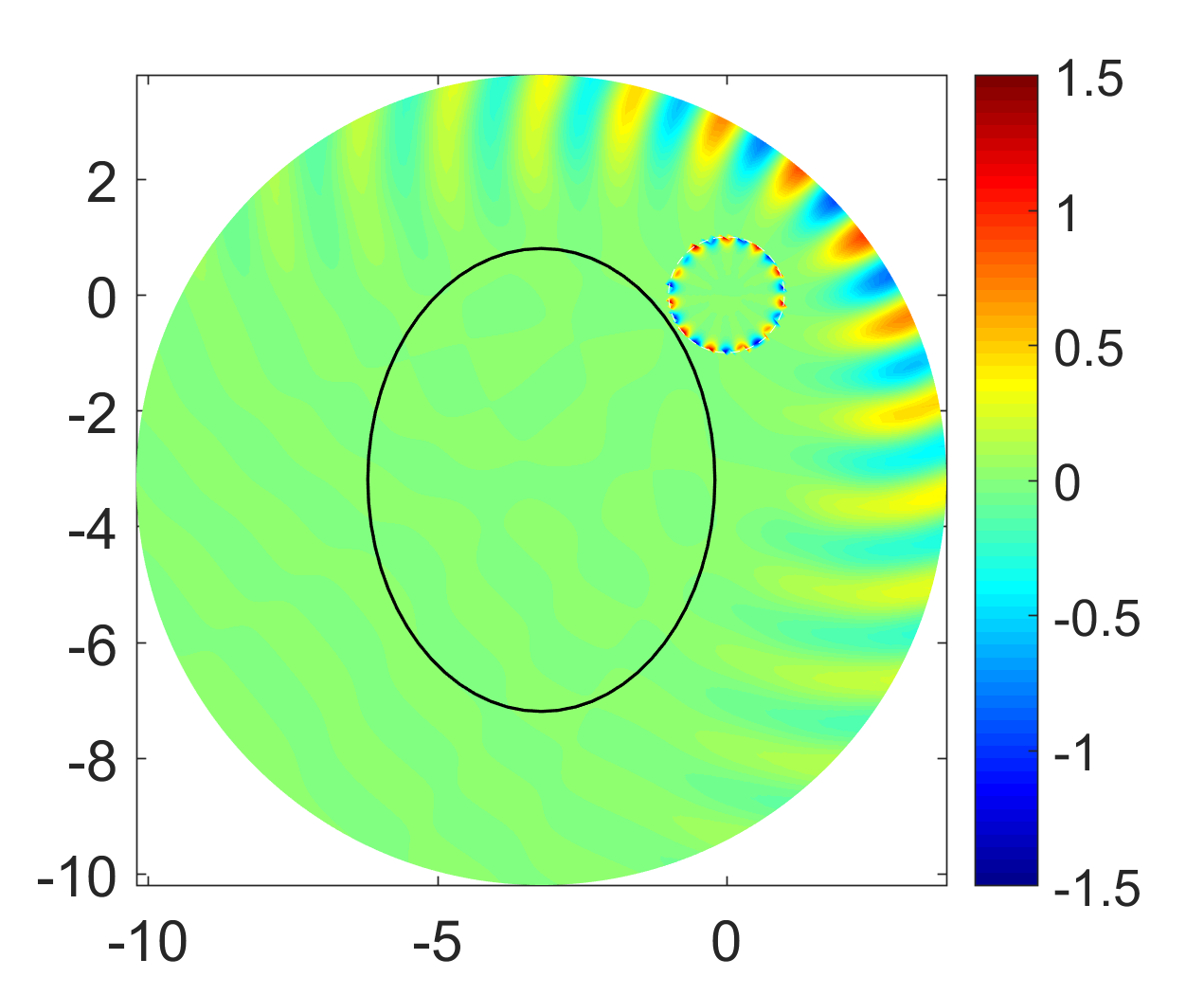}}\hfill
\hfill\subfigure[$|\nabla u|$]{\includegraphics[width=0.33\textwidth]
                   {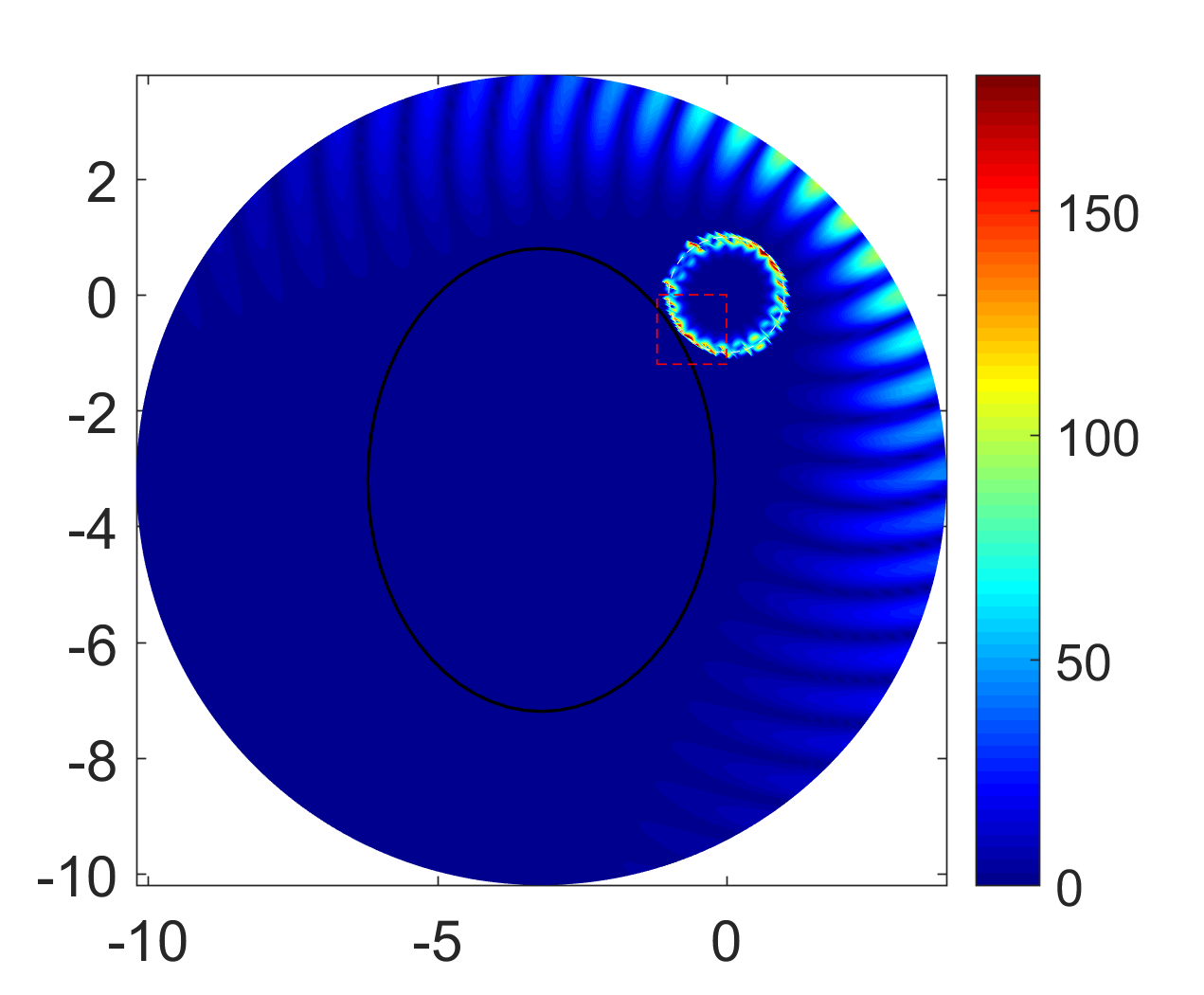}}\hfill
\hfill\subfigure[local $|\nabla u|$]{\includegraphics[width=0.33\textwidth]
                   {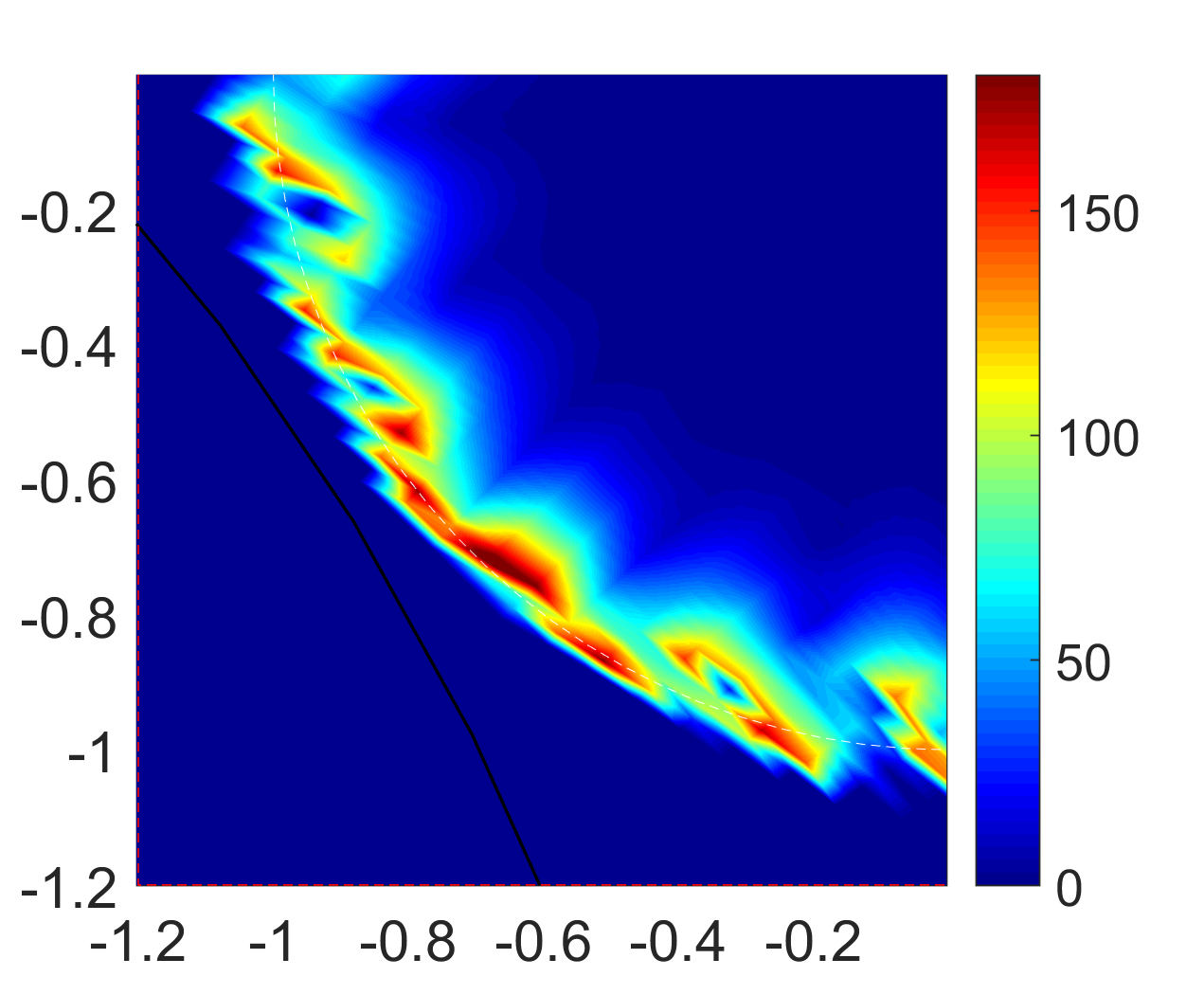}}\hfill
			\caption{The total field $u$ and the magnitude of gradient $\nabla u$ of an ellipse domain. }\label{fig:2}
		\end{minipage}
		\begin {minipage}{1\textwidth}
\hfill\subfigure[$u$]{\includegraphics[width=0.33\textwidth]
                   {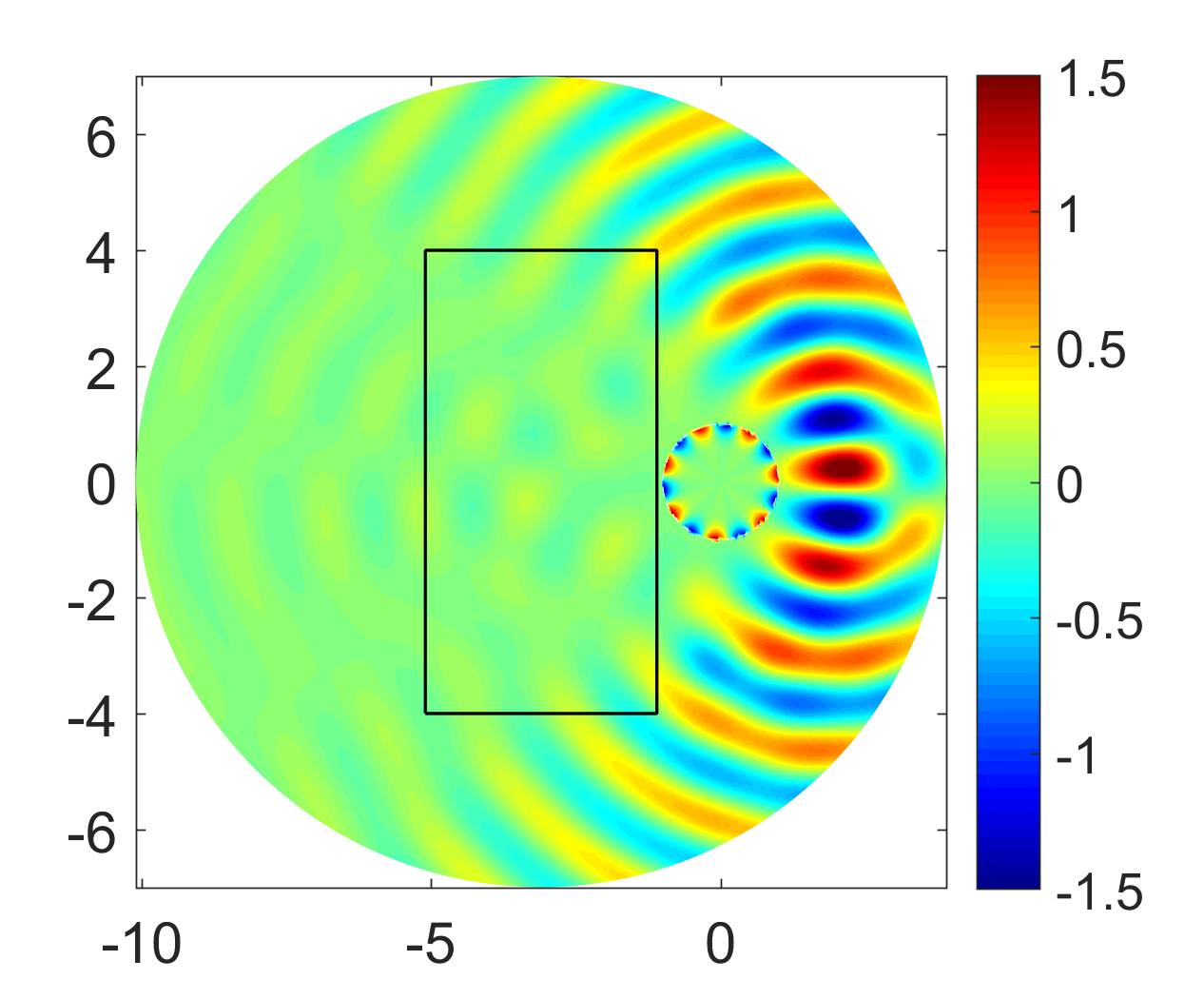}}\hfill
\hfill\subfigure[$|\nabla u|$]{\includegraphics[width=0.33\textwidth]
                   {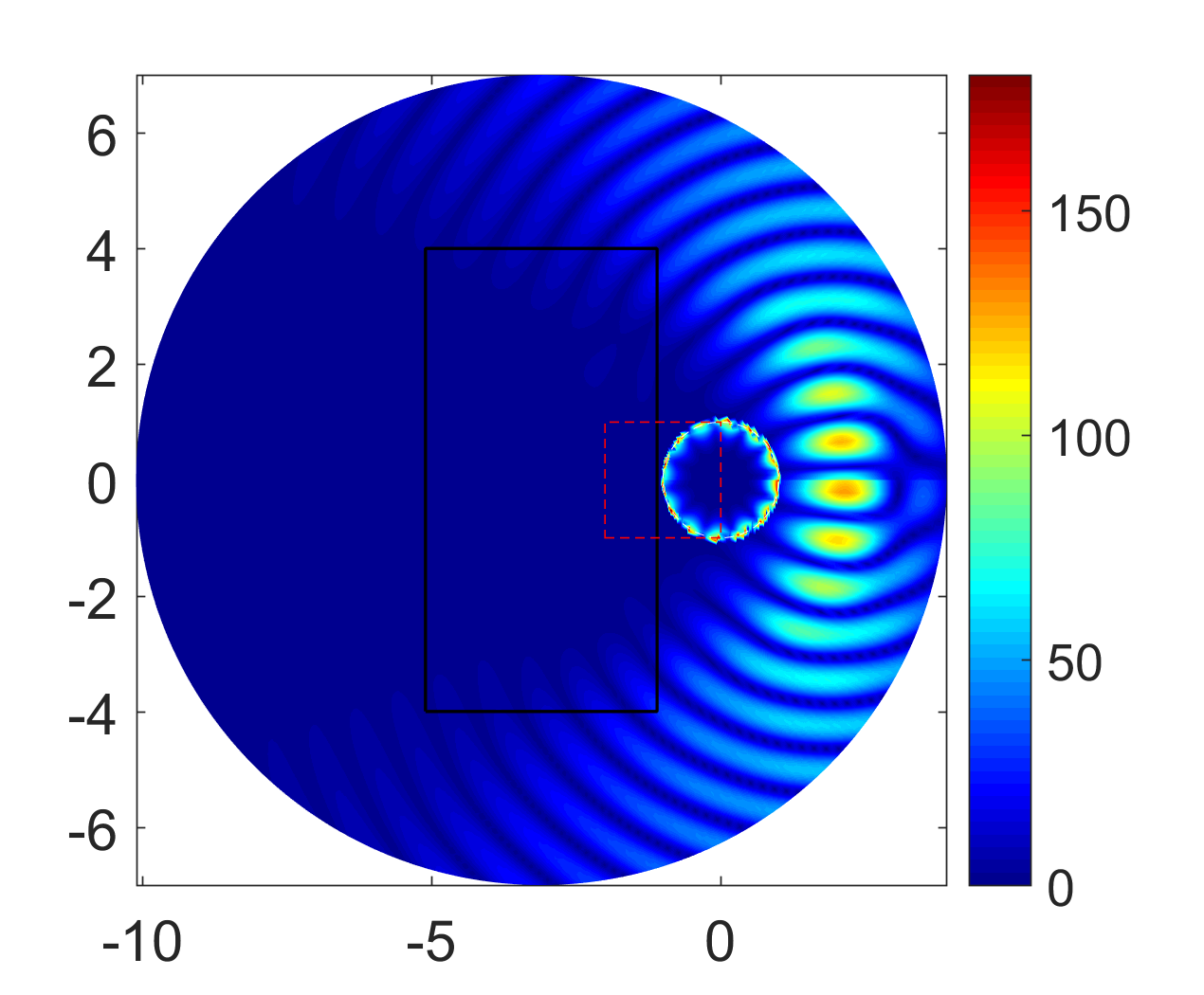}}\hfill
\hfill\subfigure[local $|\nabla u|$]{\includegraphics[width=0.33\textwidth]
                   {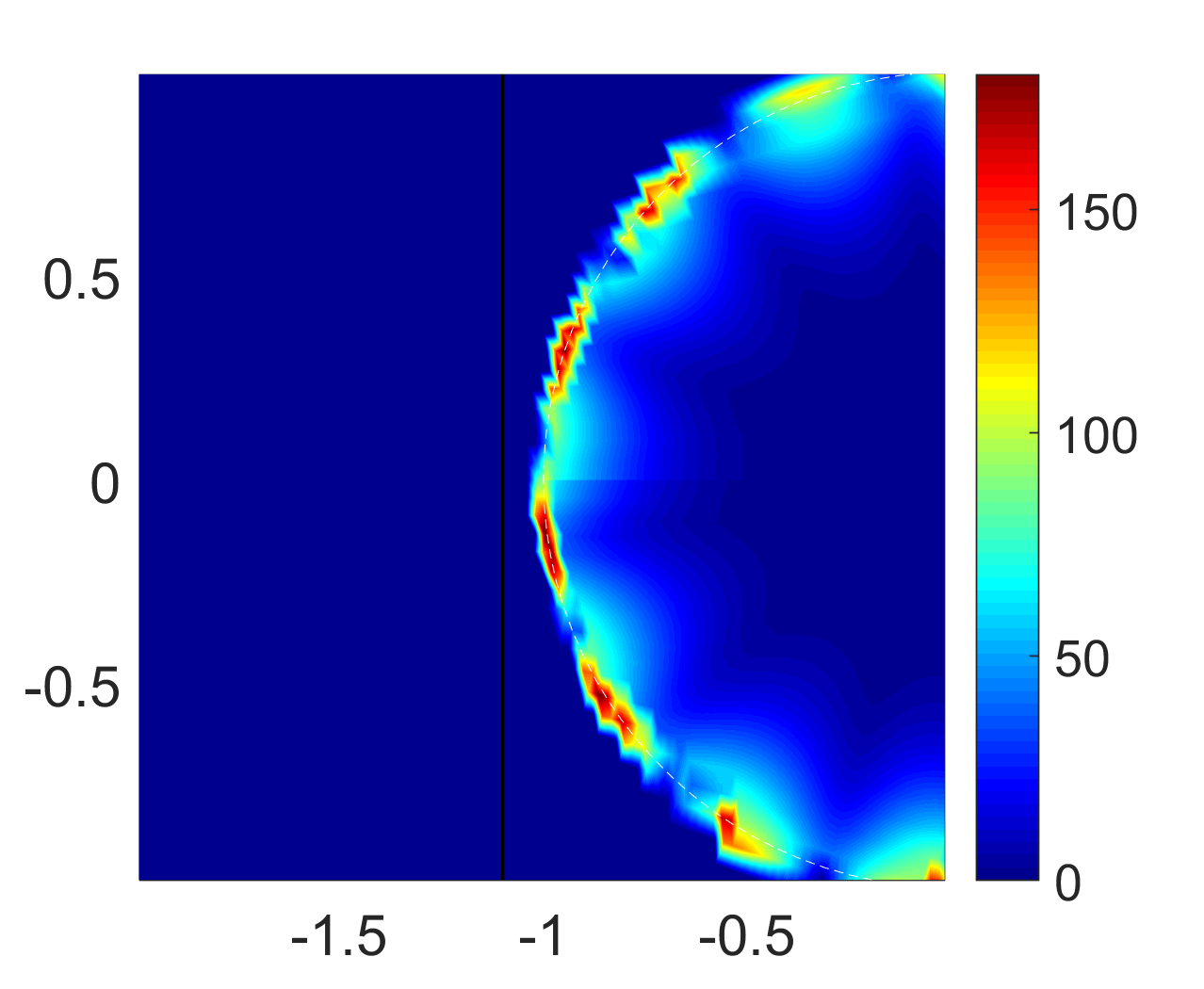}}\hfill
		\caption{The total field $u$ and the magnitude of gradient $\nabla u$ of a square domain.}\label{fig:3}
	\end{minipage}
	\begin {minipage}{1\textwidth}
\hfill\subfigure[$u$]{\includegraphics[width=0.33\textwidth]
                   {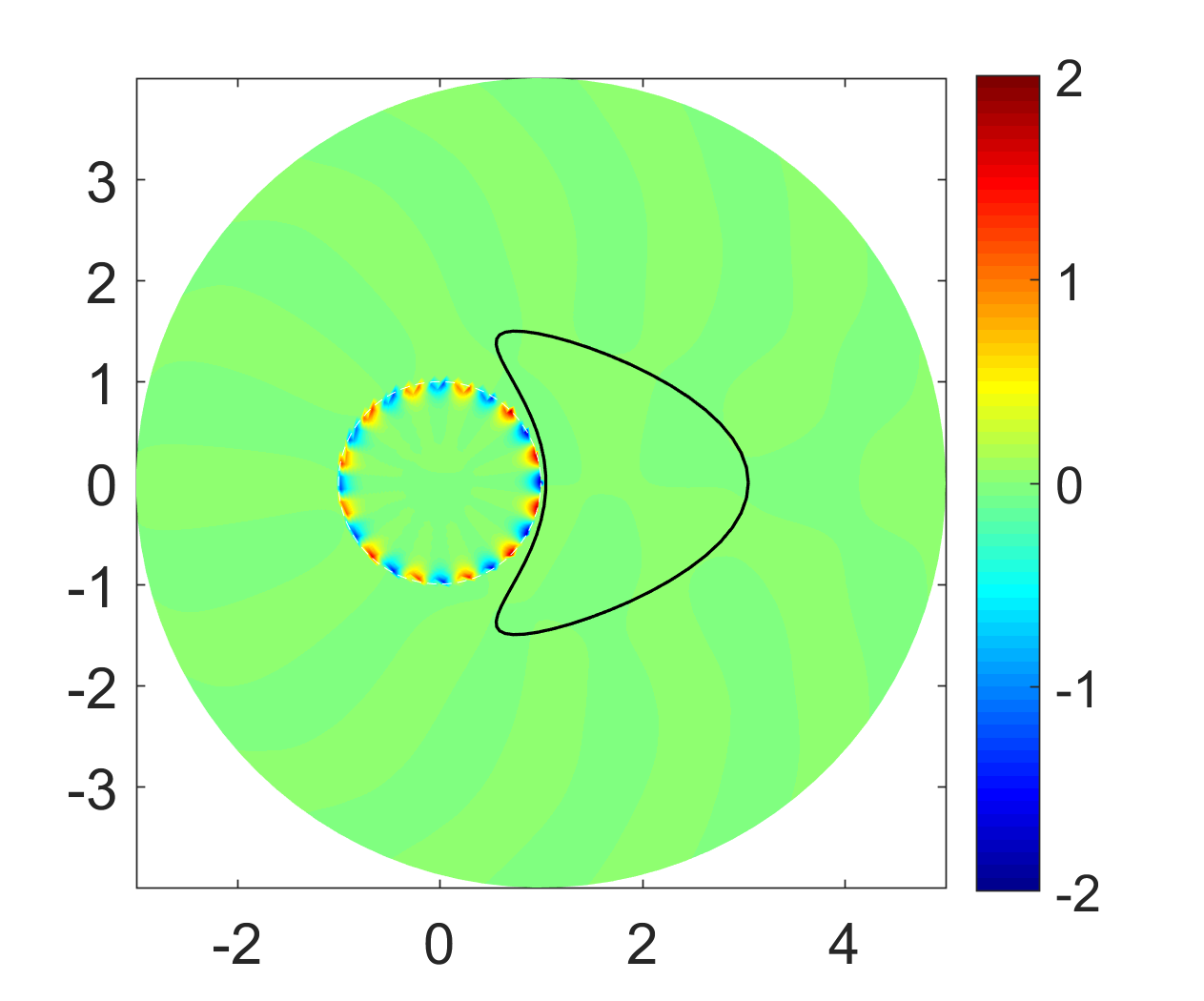}}\hfill
\hfill\subfigure[$|\nabla u|$]{\includegraphics[width=0.33\textwidth]
                   {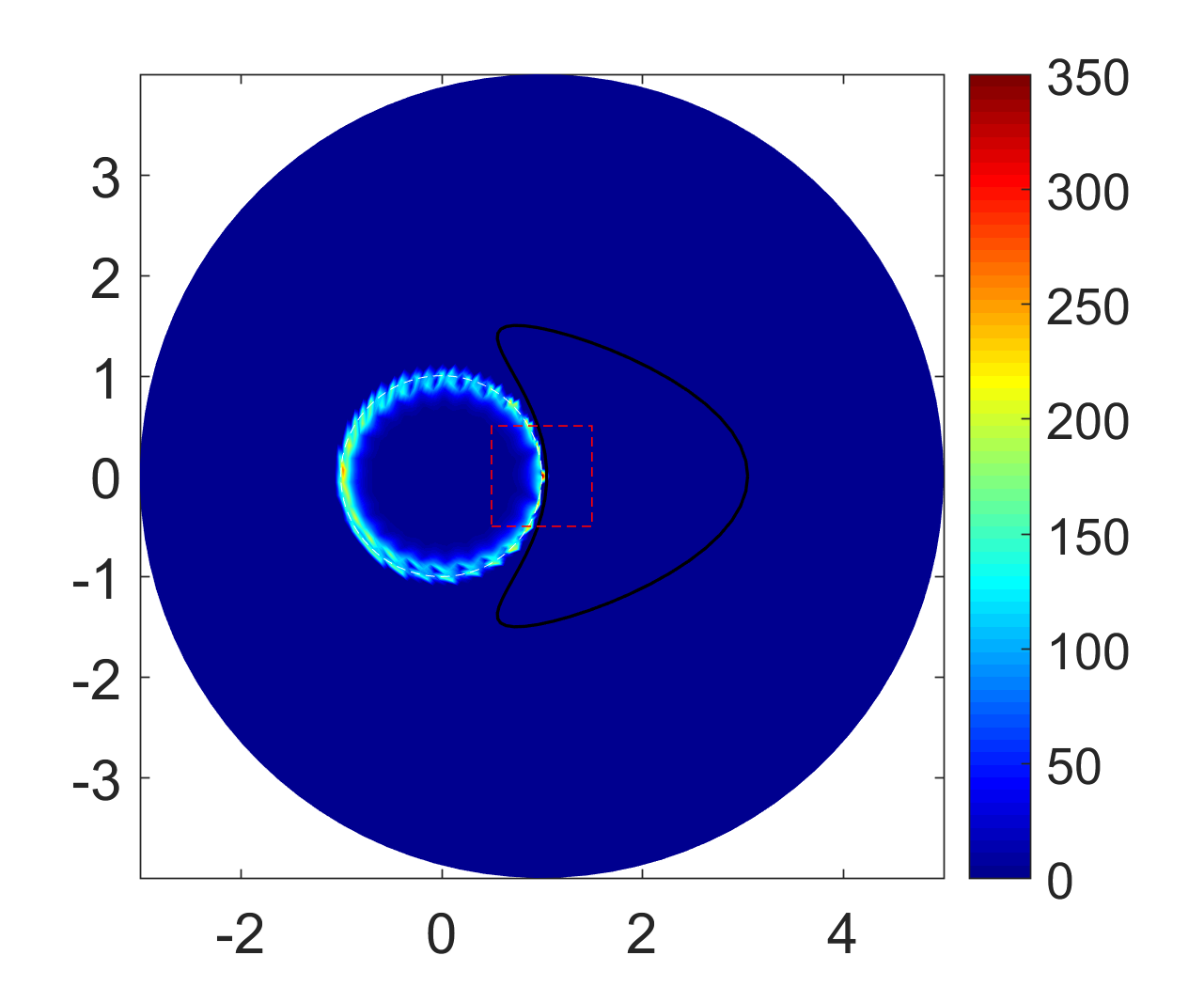}}\hfill
\hfill\subfigure[local $|\nabla u|$]{\includegraphics[width=0.33\textwidth]
                   {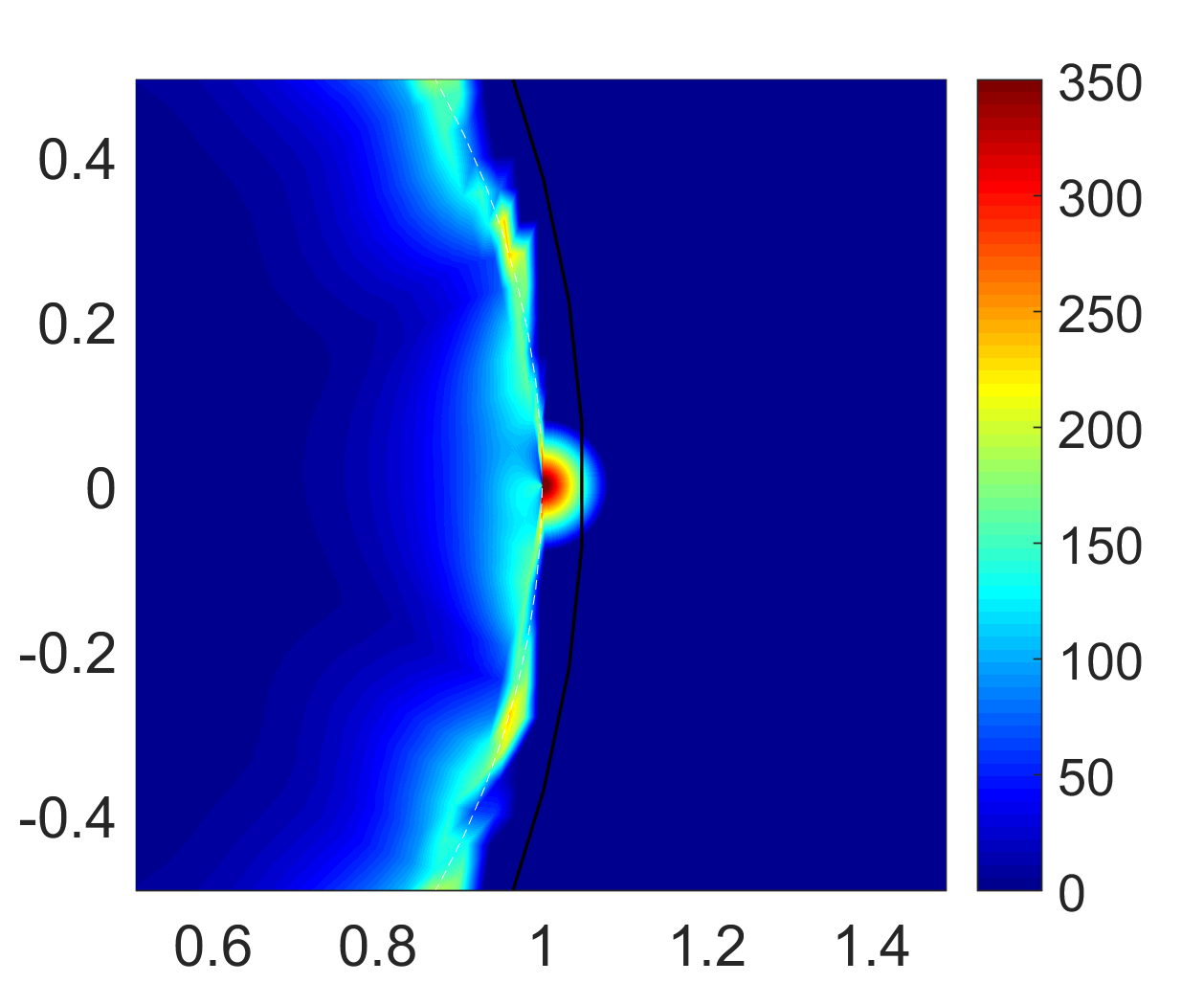}}\hfill
		\caption{The total field $u$ and the magnitude of gradient $\nabla u$ of a kite domain.}\label{fig:4}
	\end{minipage}
	\end{figure}

 From figures  \ref{fig:2}(a), \ref{fig:3}(a) and \ref{fig:4}(a),  one can find that the total field  is bounded and is a sufficiently precise approximation to the surface-localized eigenfunction and trivial eigenfunction in the area $B$ and $\Omega$ respectively. Furthermore,  it can be observed that the gradient field $\nabla u$ blows up near the scatterer and would decay sharply when the total field leaves the area B, as depicted in figures \ref{fig:2}(b), \ref{fig:3}(b) and \ref{fig:4}(b). To exhibit the imaging results clearly, we enlarge a portion of the image (the red dashed square domain in figures \ref{fig:2}(b), \ref{fig:3}(b) and \ref{fig:4}(b)). It is clear to see that the gradient of the total field demonstrates the field concentration phenomenon in figures \ref{fig:2}(c),  \ref{fig:3}(c) and \ref{fig:4}(c). It is worth mentioning that the dashed circles presented in the figures are virtual objects and do not exist in practical applications. Moreover, there are no restrictions on the scale of the circles, and they can be of arbitrary scale. Therefore, our method can demonstrate good performance for multi-scaled scatterers.

\section{surface-oscillation of transmission eigenfunctions}\label{oscillation}
In general,  the magnitude of  a wave field gradient  is proportional to its wavenumber,  as observed in plane waves  such as the plane wave $e^{\mi k \bx \cdot d }$ where $d \in \mathbb{S}^{N-1}$. However, this relationship does not hold for surface-localized transmission resonant eigenfunctions. The oscillating behaviors of these transmission eigenfunctions tends to concentrate in the neighborhood along the boundary and the oscillating frequencies are much higher than the excitation frequency.

\subsection{Auxiliary results}
Before proving Theorem \ref{oscillating}, we first introduce several auxiliary lemmas to facilitate the main proof.
\begin{Lemma}\label{besselintegral}
The following integral formula for any $m>0$ and $y \geqslant 0$ holds:
\begin{equation*}
\int_{0}^y J'^2_m(x)x + \frac{m^2}{x}J_m^2(x) \dx{x} = \int_0^y J^2_{m-1}(x)x\dx{x} - m J_m^2(y).
\end{equation*}
\end{Lemma}
\begin{proof}
The recurrence $J'_m(x) = J_{m-1}(x)-m/xJ_m(x)$ gives
\begin{align*}
\int_{0}^y J'^2_m(x)x + \frac{m^2}{x}J_m^2(x) \dx{x} &= \int_{0}^y \left( J_{m-1}(x)- \frac{m}{x}J_m(x)\right)^2x + \frac{m^2}{x}J_m^2(x) \dx{x} \\
&= \int_{0}^y J^2_{m-1}(x)x -2mJ_m(x)J'_m(x)\dx{x} .
\end{align*}
Then using the integration by part again we get the equality.
\end{proof}

\begin{Lemma}\cite{Lohoefer1998}\label{legendre}
For real $x \in [-1,1]$ and integer $l,m$ with $1 \leqslant l \leqslant m$, it holds that
\begin{equation*}
\frac{1}{\sqrt{2.22(l+1)}}  < \max_{x \in [-1,1]} \left| P(m,l;x) \right|:=\max_{x \in [-1,1]} \left|\sqrt{\frac{(m-l)!}{(m+l)!}}P_m^l(x)\right|  < \frac{2^{5/4}}{\pi^{3/4}}\frac{1}{l^{1/4}}.
\end{equation*}
There exists an absolute maximum $x_0\in [0,1)$ such that $|P_m^l(x_0)| =  \max\limits_{x\in [-1,1]}|P_m^l(x)| $ satisfying
\begin{equation*}
1 - \frac{(1.11(l+1))^2}{(m+1/2)^2} \leqslant x_0 \leqslant 1- \frac{l^2}{m(m+1)}.
\end{equation*}
\end{Lemma}

\subsection{proof of Theorem \ref{oscillating}}
\begin{proof}
The existence of a sequence of transmission eigenfunctions and their surface-localization have been proved in Lemma \ref{31}. We focus on the surface oscillating behaviors of transmission eigenfunctions  here. The transmission eigenfunctions share the same mathematical representations for any given radius $r_0$ and the proof is similar. For simplicity, we assume $r_0 = 1$.

The transmission eigenfunctions $(v_m,w_m)$ in $\bbr^2$ are given in the formula \eqref{eigenfunction2}. The gradient formula in polar coordinates yields
\begin{align*}
\nabla v_m &=  \beta_m \left(kJ'_m(kr)\hat{r} +\frac{\mathrm{i}m}{r}J_m(kr)\hat{\theta}\right)e^{\mathrm{i}m\theta} , \\
\nabla w_m &=  \alpha_m \left(kn_mJ'_m(kn_mr)\hat{r} +\frac{\mathrm{i}m}{r}J_m(kn_mr)\hat{\theta}\right)e^{\mathrm{i}m\theta} ,
\end{align*}
Where $(\hat{r},\hat{\theta})$ is a pair of unit orthonormal vector in polar coordinates.
By using the orthogonality of $\{e^{\mi m \theta}\}_{m \in \bbn}$ in the unit circle, one has
\begin{align*}
\|\nabla v_m\|_{L^2(B_\xi)}^2 =\beta_m^2\int_0^{k\xi}  J'^2_m(r)r +\frac{m^2}{r}J^2_m(r) \dx{r}.
\end{align*}
By using the asymptotic formulas of Bessel functions \eqref{Besselfunction} when $m$ is sufficiently large, it holds that
\begin{align*}
\frac{\|\nabla v_m\|_{L^2(B_\xi)}^2}{\|\nabla v_m\|_{L^2(B_1)}^2} = \frac{\int_0^{k\xi}  J'^2_m(r)r +\frac{m^2}{r}J^2_m(r) \dx{r}}{\int_0^{k}  J'^2_m(r)r +\frac{m^2}{r}J^2_m(r) \dx{r}}  \sim \frac{\int_0^{k\xi} r^{2m-1} \dx{r}}{\int_0^{k} r^{2m-1} \dx{r}} \sim \xi^{2m}\rightarrow 0 \coma m \rightarrow \infty.
\end{align*}
So $\{\nabla v_m\}$ are also surface-localized especially for sufficiently $m$. For $w_m$,
\begin{align*}
\|\nabla w_m\|_{L^2(B_\xi)}^2 &= \alpha_m^2 \int_0^\xi \left( k^2n_m^2J'^2_m(kn_mr) + \frac{m^2}{r^2}J^2_m(kn_mr)\right)r \dx{r} \\
&= \alpha_m^2 \int_0^{kn_m\xi}  J^2_{m-1}(r)r \dx{r} - mJ_m^2(kn_m\xi).
\end{align*}
where we have used the Lemma \ref{besselintegral}.
By following a similar argument as \eqref{wmestimate}, it holds that
\begin{equation*}
\int_0^{kn_m\xi}  J^2_{m-1}(r)r \dx{r} - mJ_m^2(kn_m\xi) \leqslant  \frac{1}{2}J^2_{m-1}(kn_m\xi) (kn_m\xi)^2,
\end{equation*}
and
\begin{equation*}
\begin{aligned}
\int_0^{kn_m}  J^2_{m-1}(r)r \dx{r} &- mJ_m^2(kn_m\xi)  \geqslant \int_0^{m-1}  J^2_{m-1}(r)r \dx{r} - mJ_m^2(m-1) \\
&\geqslant  \frac{1}{2}\frac{(m-1)^2J_{m-1}^3(m-1)}{J_{m-1}(m-1)+2(m-1)J'_{m-1}(m-1)}-  mJ_{m-1}^2(m-1),
\end{aligned}
\end{equation*}
where we have used the fact $J_m(m-1) <J_{m-1}(m-1)$. Hence,
\begin{equation}\label{wmm}
\begin{aligned}
\frac{\|\nabla w_m\|_{L^2(B_\xi)}^2}{\|\nabla w_m\|_{L^2(B_1)}^2} &=  \frac{\int_0^{kn_m\xi}  J^2_{m-1}(r)r \dx{r} - mJ_m^2(kn_m\xi)}{ \int_0^{kn_m}  J^2_{m-1}(r)r \dx{r} - mJ_m^2(kn_m)} \\
&\leqslant \frac{J^2_{m-1}(kn_m\xi) (kn_m\xi)^2}{\frac{(m-1)^2J_{m-1}^3(m-1)}{J_{m-1}(m-1)+2(m-1)J'_{m-1}(m-1)}-  2mJ_{m-1}^2(m-1)} \\
&\leqslant \frac{J^2_{m-1}(kn_m\xi)}{J_{m-1}^2(m-1)}\frac{ (kn_m\xi)^2}{(m-1)^2}\frac{1}{\frac{J_{m-1}(m-1)}{J_{m-1}(m-1)+2(m-1)J'_{m-1}(m-1)}-\frac{2m}{(m-1)^2}}.
\end{aligned}
\end{equation}

From the asymptotic formula for $J_{m-1}(m-1)$ and $J'_{m-1}(m-1)$ (see 9.3.31 in \cite{Abramowitz1988}), we can obtain
\begin{align*}
\frac{J_{m-1}(m-1)}{J_{m-1}(m-1)+2(m-1)J'_{m-1}(m-1)}-\frac{2m}{(m-1)^2}
= (m-1)^{-\frac{2}{3}} \left(C + (m-1)^{-\frac{1}{3}} \right) ,
\end{align*}
where $C$ is some positive constant.
Note that $kn_m\xi <(m-1)$ and following a similar argument as \eqref{wmfinal}, we can readily get
\begin{align*}
\frac{\|\nabla w_m\|_{L^2(B_\xi)}}{\|\nabla w_m\|_{L^2(B_1)}}   &\leqslant C (m-1)^{\frac{2}{3}}Ai((m-1)^\frac{2}{3}\zeta_{\xi}) \leqslant C (m-1)^{\frac{1}{2}} e^{-\frac{2}{3}\zeta_{\xi}^{3/2}(m-1)}  \rightarrow 0 \coma m \rightarrow \infty.
\end{align*}
where $C$ is some constant  only depending on $k,\xi$. Then $\{\nabla w_m\}_{m \in \bbn}$ are also surface-localized.

We now prove the oscillating frequencies of $(v_m,w_m)$ are much higher than the excitation frequency $k$.
The normalized condition $\|v_m\|_{L^2(B_1)} = 1$ shows that
\begin{equation*}
\begin{aligned}
\nabla v_m =&   \beta_m \left(kJ'_m(kr)\hat{r} +\frac{\mathrm{i}m}{r}J_m(kr)\hat{\theta}\right)e^{\mathrm{i}m\theta} \\
=& \frac{1}{\sqrt{2\pi}}\frac{kJ_m(k)}{\sqrt{\int_0^{k}J_m^2(r)r\dx{r}}} \left(k\frac{J'_m(kr)}{J_m(k)}\hat{r} +\frac{\mathrm{i}m}{r}\frac{J_m(kr)}{J_m(k)}\hat{\theta}\right)e^{\mathrm{i}m\theta}
\end{aligned}
\end{equation*}
First, the formula of \eqref{Besselfunction} yields
\begin{equation}\label{zq}
\frac{kJ_m(k)}{\sqrt{\int_0^{k}J_m^2(r)r\dx{r}}} \sim \frac{k^{m+1}}{\sqrt{\frac{1}{2(m+1)}k^{2m+2}}} \sim 2\sqrt{m+1}.
\end{equation}
Moreover, The Bessel function $J_m(x)$ and its derivative  $J'_m(x)$ are strictly monotonously increasing. Then $\nabla v_m$ attains the maximum at $r=1$. By using the asymptotic formula \eqref{Besselfunction} again, it readily holds that
\begin{equation}\label{lower}
\frac{\|\nabla v_m\|_{L^\infty(\Sigma_\xi)}}{k} \geqslant C_{k,s_0} m^{\frac{3}{2}} \rightarrow \infty \quad \text{as} \quad  m \rightarrow \infty.
\end{equation}
It is remarked that the lower bound \eqref{lower} is in fact optimal and $\nabla v_m$ attains the supremum on the boundary $r=1$.

For $\nabla w_m$, the relation between $\alpha_m$ and $\beta_m$ together with the  normalized condition $\|v_m\|_{L^2(B_1)} = 1$ gives
\begin{align*}
\nabla w_m &=  \alpha_m \left(kn_mJ'_m(kn_mr)\hat{r} +\frac{\mathrm{i}m}{r}J_m(kn_mr)\hat{\theta}\right)e^{\mathrm{i}m\theta} \\
&= \frac{1}{\sqrt{2\pi}}\frac{kJ_m(k)}{\sqrt{\int_0^{k}J_m^2(r)r\dx{r}}} \left(kn_m\frac{J'_m(kn_mr)}{J_m(kn_m)}\hat{r} +\frac{\mathrm{i}m}{r}\frac{J_m(kn_mr)}{J_m(kn_m)}\hat{\theta}\right)e^{\mathrm{i}m\theta}.
\end{align*}
The maximum formula of Bessel functions and their derivatives
$$\max_{x\in \bbr} J_m(x) = J_m(j'_{m,1}) \sim m^{-\frac{1}{3}} \coma \max_{x\in \bbr}{J'_m(x)} \sim m^{-\frac{2}{3}}, 
$$
together with \eqref{216}, \eqref{zq} yields
\begin{align*}
\frac{\|\nabla v_m\|_{L^\infty(\Sigma_\xi)}}{k} &= \frac{k}{\sqrt{2\pi}}\frac{e^{im\theta}}{\sqrt{\int_0^{k}J_m^2(r)r\dx{r}}}\frac{J_m(k)}{J_m(kn_m)}\left\|n_mJ_m'(kn_mr)\hat{r}+ \frac{\mathrm{i}m}{rk}J_m(kn_mr)\hat{\theta} \right\|_{L^\infty(\Sigma_\xi)} \\
&\geqslant C_{k,s_0} m^{\frac{3}{2}} \rightarrow \infty \quad \text{as} \quad m \rightarrow \infty.
\end{align*}
It is remarked that  $\nabla w_m$ attains the maximum at  $r_m :=j'_{m,1}/{kn_m}$ where $r_m \rightarrow 1$ as $m\rightarrow \infty$. This is a circle with the radius $r_m$ which is very close but not touch the boundary $r=1$.

The entire proof in three dimensions follows a similar structure to that in two dimensions, with some necessary modifications due to technical calculations. We will outline the required modifications in the subsequent sections.
The corresponding gradient $\nabla v_m$ in spherical coordinate is
\begin{align*}
\nabla v_m^l = \frac{\partial v_m^l}{\partial r}\hat{r} + \nabla_{(\theta,\psi)} v_m^l = \beta_m^l\left( k j'_m(kr)Y_m^l\hat{r} +  j_{m+\frac{1}{2}}(kr) \nabla_{(\theta,\psi)}Y_m^l\right).
\end{align*}
Since ${r}{(m(m+1)})^{-1/2}\nabla_{(\theta,\psi)}Y_m^l$ are a sequence of  orthonormal  vector spherical harmonics, then
\begin{align*}
\|\nabla v_m^l \|^2_{L^2(B_\xi)} =& (\beta_m^l)^2 \left(\int_0^\xi j'^2_m(kr)k^2r^2 + m(m+1)j^2_{m}(kr) \dx{r}\right) \\
&=\frac{(\beta_m^l)^2}{k} \left(\int_0^{k\xi} j'^2_m(r)r^2 + m(m+1)j^2_{m}(r) \dx{r}\right) .
\end{align*}
Then
\begin{align*}
\frac{\|\nabla v_m^l \|^2_{L^2(B_\xi)}}{\|\nabla v_m^l \|^2_{L^2(B_1)}} &= \frac{\int_0^{k\xi} j'^2_m(r)r^2 + m(m+1)j^2_{m}(r) \dx{r}}{\int_0^{k} j'^2_m(r)r^2 + m(m+1)j^2_{m}(r) \dx{r}}\\
&\sim \frac{\int_0^{k\xi}r^{2m}\dx{r}}{\int_0^{k}r^{2m}\dx{r}} \sim \xi^{2m+1}\rightarrow 0 \quad \text{as} \quad m \rightarrow \infty.
\end{align*}
Similarly, the representation of $\nabla w_m^l$ is
\begin{align*}
\nabla w_m^l = \frac{\partial w_m^l}{\partial r}\hat{r} + \nabla_{(\theta,\psi)} w_m^l = \alpha_m^l\left( kn_m j'_m(kn_mr)Y_m^l\hat{r} +  j_{m}(kn_mr) \nabla_{(\theta,\psi)}Y_m^l\right).
\end{align*}
The $L^2$-norm of $\nabla w_m^l$ is given by
\begin{align*}
\|\nabla w_m^l \|^2_{L^2(B_\xi)} =& (\alpha_m^l)^2 \left(\int_0^\xi j'^2_m(kn_mr)k^2n_m^2r^2 + m(m+1)j^2_{m}(kr) \dx{r}\right) \\
&= \frac{\pi}{2}\frac{(\alpha_m^l)^2}{kn_m} \left(\int_0^{kn_m\xi} J'^2_{m+\frac{1}{2}}(r)r + \frac{(m+\frac{1}{2})^2}{r}J^2_{m+\frac{1}{2}}(r) \dx{r} - \frac{1}{2}J^2_{m+\frac{1}{2}}(kn_m\xi)\right) \\
&=\frac{\pi}{2}\frac{(\alpha_m^l)^2}{kn_m} \left( \int_0^{kn_m\xi} J^2_{m-\frac{1}{2}}(r)r\dx{r} - (m+1)J^2_{m+\frac{1}{2}}(kn_m\xi)  \right).
\end{align*}
By using the result in \eqref{wmm}, it holds that
\begin{align*}
\frac{\|\nabla w_m^l \|^2_{L^2(B_\xi)}}{\|\nabla w_m^l \|^2_{L^2(B_1)}} &=\frac{\int_0^{kn_m\xi} J^2_{m-\frac{1}{2}}(r)r\dx{r} - (m+1)J^2_{m+\frac{1}{2}}(kn_m\xi)}{\int_0^{kn_m} J^2_{m-\frac{1}{2}}(r)r\dx{r} - (m+1)J^2_{m+\frac{1}{2}}(kn_m\xi)} \\
&\leqslant  \frac{\int_0^{kn_m\xi} J^2_{m-\frac{1}{2}}(r)r\dx{r} - (m+1)J^2_{m+\frac{1}{2}}(kn_m\xi)}{\int_0^{m -\frac{1}{2}} J^2_{m-\frac{1}{2}}(r)r\dx{r} - (m+1)J^2_{m-\frac{1}{2}}(m-\frac{1}{2})} \\
&= \frac{ J^2_{m-\frac{1}{2}}(kn_m\xi)}{J_{m-\frac{1}{2}}^2(m-\frac{1}{2})}\frac{(kn_m\xi)^2}{(m-\frac{1}{2})^2}\frac{1}{\frac{J_{m-\frac{1}{2}}(m-\frac{1}{2})}{J_{m-\frac{1}{2}}(m-\frac{1}{2})+2(m-\frac{1}{2}) J'_{m-\frac{1}{2}}(m-\frac{1}{2})} - 2 \frac{m+1}{(m-\frac{1}{2})^2}} \\
&\longrightarrow 0 \quad \text{as} \quad m \rightarrow \infty.
\end{align*}

In order to prove the oscillating frequencies of $(v_m^l,w_m^l)$ are much higher than the excitation frequency $k$, we need a clearer representation for the transmission eigenfunction. First,  The normalized spherical harmonics $Y_m^l$ has the explicit expression,
\begin{equation*}
Y_m^l(\theta,\varphi) = \sqrt{\frac{2m+1}{4\pi}\frac{(m-|l|)!}{(m+|l|)!}}P_m^{|l|}(\cos \theta) e^{\mathrm{i}l\varphi} \coma \theta \in [0,\pi] , \varphi \in [0,2\pi) , -m\leqslant l \leqslant m,
\end{equation*}
where $P_m^{|l|}$ are the associated Legendre polynomials.  The normalized condition $\|v_m^l\|_{L^2(B_1)}=1$ gives
\begin{equation*}
\beta_m^l =  \frac{1}{\sqrt{\int_0^1 j^2_{m}(kr)r^2 \dx{r}}}  = \frac{2}{\pi}\frac{k}{\sqrt{\int_0^k J^2_{m+\frac{1}{2}}(r)r \dx{r}}}.
\end{equation*}
Moreover,  the corresponding gradient $\nabla v_m^l$ in spherical coordinates is
\begin{equation}\label{nablavm}
\begin{aligned}
\nabla v_m^l =& \frac{\partial v_m^l}{\partial r}\hat{r} +\frac{1}{r} \frac{\partial v_m^l}{\partial \theta}\hat{\theta} + \frac{1}{r \sin \theta} \frac{\partial v_m^l}{\partial \varphi} \hat{\varphi} \nonumber\\
=&\beta_m^l\sqrt{\frac{2m+1}{4\pi}\frac{(m-|l|)!}{(m+|l|)!}}r^{-3/2}e^{\mathrm{i}l\varphi}
\left[ \left(-\frac{1}{2}J_{m+\frac{1}{2}}(kr)+krJ'_{m+\frac{1}{2}}(kr) \right) P_m^{|l|}(\cos \theta)\hat{r} \right.  \nonumber\\
& \left.  + J_{m+\frac{1}{2}}(kr)\frac{\dx{P_m^{|l|}(\cos \theta)}}{\dx{\theta}}\hat{\theta} + J_{m+\frac{1}{2}}(kr)P_m^{|l|}(\cos \theta)\frac{\mi l}{\sin \theta}\hat{\varphi} \right]  \nonumber\\
:=& \beta_m^lJ_{m+\frac{1}{2}}(k)\sqrt{\frac{2m+1}{4\pi}} \left(I_1\hat{r} +I_2\hat{\theta} + I_3\hat{\varphi} \right),
\end{aligned}
\end{equation}

where $\hat{r} , \hat{\theta} ,  \hat{\varphi}$ are the orthonormal unit vectors in spherical coordinates.



Next, we need to estimate $I_i,i=1,2,3$ in the above formula. By using Lemma \ref{legendre}, we first have
\begin{equation}\label{i1}
\begin{aligned}
\sup_{x \in \Sigma_\xi}I_1&= \sup_{\bx \in  \Sigma_\xi} \sqrt{\frac{(m-|l|)!}{(m+|l|)!}}\frac{1}{J_{m+\frac{1}{2}}(k)} \left| r^{-\frac{3}{2}}e^{\mathrm{i}l\varphi}\Big(krJ'_{m+\frac{1}{2}}(kr)-\frac{1}{2}J_{m+\frac{1}{2}}(kr)\Big)P_m^{|l|}(\cos \theta) \right|\\
&\geqslant  C\frac{m+1/2}{(|l|+1)^{1/2}} \coma |l| \geqslant 1,
\end{aligned}
\end{equation}
where we used the asymptotic formula of Bessel functions and its derivatives for large order.
It is clear that \eqref{i1} also holds for $l=0$ since the Legendre polynomials $\max_{x \in [-1,1]} P_m(x) =1$. For $I_3$, we shall calculate the value at maximum point $x_0$ to get the lower bound. Let $x_0 =\cos\theta_0$, it holds from Lemma \ref{legendre} that
\begin{equation*}
\sin\theta_0 = \sqrt{1-x_0^2} \leqslant \frac{1.11(|l|+1)}{m+1/2} \coma 1 \leqslant |l|\leqslant m.
\end{equation*}
Then
\begin{equation}\label{i3}
\begin{aligned}
\sup_{\bx \in \Sigma_\xi} I_3 &\geqslant \sup_{x \in \Sigma_\xi} \sqrt{\frac{(m-|l|)!}{(m+|l|)!}}\frac{1}{J_{m+\frac{1}{2}}(k)} \left| r^{-\frac{3}{2}}e^{\mathrm{i}l\varphi}J_{m+\frac{1}{2}}(kr)(\sin \theta_0)^{-1}lP_m^{|l|}(\cos \theta_0)\right| \\
&\geqslant C\frac{m+1/2}{(|l|+1)^{1/2}}\frac{|l|}{|l|+1} \coma |l| \geqslant 1,
\end{aligned}
\end{equation}
and $\sup_{\bx \in \Sigma_\xi} I_3 =0$ when $l=0$.

We shall estimate $I_2$ in several cases. When $2 \leqslant |l| \leqslant m-1$, we estimate the upper bound of the derivative of the associated Legendre polynomial. It holds that
\begin{equation}\label{a1}
\begin{aligned}
\sup_{\theta \in [0,\pi]} &\left|\sqrt{\frac{(m-|l|)!}{(m+|l|)!}}\frac{\dx{P_m^{|l|}(\cos \theta)}}{\dx{\theta}}\right| =  \sup_{\theta \in [0,\pi]} \left|\sqrt{\frac{(m-|l|)!}{(m+|l|)!}}\frac{\dx{P_m^{|l|}(\cos \theta)}}{\dx{(\cos \theta})} \sin \theta \right| \\
&=  \sup_{\theta \in [0,\pi]}\frac{1}{2}\sqrt{\frac{(m-|l|)!}{(m+|l|)!}}\left|(m+|l|)(m-|l|+1)P_{m}^{|l|-1}(\cos \theta) - P_{m}^{|l|+1}(\cos\theta) \right| \\
&=\sup_{\theta \in [0,\pi]}\sqrt{\frac{(m-|l|)!}{(m+|l|)!}}\left|\frac{|l|\cos \theta}{\sin \theta}P_m^{|l|}(\cos \theta) -  P_{m}^{|l|+1}(\cos\theta) \right| \\
&\leqslant \sup_{\theta \in [0,\pi]}\sqrt{\frac{(m-|l|)!}{(m+|l|)!}}\left|\frac{|l|}{\sin \theta}P_m^{|l|}(\cos \theta) \right| +C (m+\frac{1}{2}) \sup_{\theta \in [0,\pi]} \left| P(m,|l|;\cos \theta)\right|.
\end{aligned}
\end{equation}
where we have used the fact  that the maximum of $P(m,|l|+1;\cos \theta)$ is uniformly less than that of $P(m,|l|;\cos \theta)$ with respect to $m$ and $l$ in the last line \cite{Lohoefer1998}.
For other cases $|l|=m,1\; \textup{and} \; 0$, straight calculations show that
\begin{align}
\sup_{\theta \in [0,\pi]} \left|\sqrt{\frac{1}{(2m)!}}\frac{\dx{P^m_m(\cos \theta)}}{\dx{\theta}}\right| &= \sqrt{2m}\sup_{\theta \in [0,\pi]}\left| P(m,m-1;\cos \theta) \right| \coma \label{a15}  \\
\sup_{\theta \in [0,\pi]} \sqrt{\frac{(m-1)!}{(m+1)!}}\left|\frac{\dx{P^1_m(\cos \theta)}}{\dx{\theta}}\right| &= \sup_{\theta \in [0,\pi]} \left|P_m(\cos \theta) - \sqrt{(m+2)(m-1)}P(m,2;\cos \theta)\right|.\label{a2}
\end{align}
Comparing the formulas \eqref{a1},\eqref{a15}, \eqref{a2} with the formula \eqref{i3}, one can find that the upper bound of $I_2$ can be uniformly controlled by that of $I_3$ when $|l| \geqslant 1$, i.e.
\begin{equation} \label{i2}
\sup_{x \in \Sigma_\xi} I_2 =\sup_{\bx \in \Sigma_\xi} \sqrt{\frac{(m-|l|)!}{(m+|l|)!}}\frac{1}{J_{m+\frac{1}{2}}(k)} \left| r^{-3/2}e^{\mathrm{i}l\varphi}J_{m+\frac{1}{2}}(kr)\frac{\dx{P_m^{|l|}(\cos \theta)}}{\dx{\theta}}\right| \leqslant C \sup_{\bx \in \Sigma_\xi} I_3  .
\end{equation}
While $l=0$, it holds that
\begin{equation}\label{i4}
\sup_{\theta \in [0,\pi]} \left|\frac{\dx{P_m(\cos \theta)}}{\dx{\theta}}\right| = \sqrt{m(m+1)} \sup_{\theta \in [0,\pi]} \left|P(m,1;\cos \theta) \right|.
\end{equation}
Combining \eqref{nablavm} , \eqref{i1}, \eqref{i3}, \eqref{i2} \eqref{i4}, one can deduce that
\begin{align*}
\frac{\left\|\nabla v_m^l\right\|_{L^\infty(\Sigma_\xi)}}{k} &=\frac{J_{m+\frac{1}{2}}(k)}{\sqrt{\int_0^1 J_{m+\frac{1}{2}}(kr)r \dx{r}}}\sqrt{\frac{2m+1}{4\pi k^2}}\left\|\sqrt{(I_1)^2 +(I_2)^2 + (I_3)^2} \right\|_{L^\infty(\Sigma_2)} \\
&\geqslant C\frac{(m+1/2)^2}{(|l|+1)^{1/2}} \geqslant C(m+\frac{1}{2})^{\frac{3}{2}} \rightarrow \infty.
\end{align*}
when $m \rightarrow \infty$. In fact, $\nabla v_m^l$ attains the supremum on the boundary sphere $r =1$.

From the relation between $\alpha_m^l$ and $\beta_m^l$, we can obtain
\begin{equation*}
\alpha_m^l = \frac{j_m(k)}{j_m(kn_m)}\beta_m^l = \sqrt{n_m}\frac{2}{\pi}\frac{J_{m+\frac{1}{2}}(k)}{J_{m+\frac{1}{2}}(kn_m)}\frac{k}{\sqrt{\int_0^k J^2_{m+\frac{1}{2}}(r)r \dx{r}}}.
\end{equation*}
We can get the coefficient of $w_m^l$ through the relation of $\alpha_m^l$ and $\beta_m^l$. Then the gradient of $w_m^l$ in spherical coordinates is
\begin{align*}
\nabla w_m^l =& \alpha_m^l \sqrt{\frac{2m+1}{4\pi}\frac{(m-|l|)!}{(m+|l|)!}}r^{-3/2}e^{il\varphi}  \Bigg[ \Big( kn_mrJ'_{m+\frac{1}{2}}(kn_mr) -\frac{1}{2}J_{m+\frac{1}{2}}(kn_mr) \Big)P_m^{|l|}(\cos \theta)\hat{r}
\nonumber \\
&  + J_{m+\frac{1}{2}}(kn_mr)\frac{\dx{P_m^{|l|}(\cos \theta)}}{\dx{\theta}}\hat{\theta} +  J_{m+\frac{1}{2}}(kn_mr)P_m^{|l|}(\cos \theta)\frac{\mi l}{\sin \theta}\hat{\varphi} \Bigg] \nonumber\\
&:=\alpha_m^l\sqrt{\frac{2m+1}{4\pi}}\left(L_1\hat{r} +L_2\hat{\theta} + L_3\hat{\varphi} \right).
\end{align*}
 We note the fact
\begin{equation*}
kn_m \in (j_{{m+\frac{1}{2}},s_0},j_{{m+\frac{1}{2}},s_0+1}) = {m+\frac{1}{2}} + a_{s_0}({m+\frac{1}{2}})^{\frac{1}{3}} + \mathcal{O}(({m+\frac{1}{2}})^{-\frac{1}{3}}) >j_{{m+\frac{1}{2}},1} > j'_{{m+\frac{1}{2}},1},
\end{equation*}
where $a_{s_0}$ is some positive constant. Then 
$$\max_{x\in \bbr} J_{m+\frac{1}{2}}(x) = J_{m+\frac{1}{2}}(j'_{{m+\frac{1}{2}},1}) \sim ({m+\frac{1}{2}})^{-\frac{1}{3}} \coma \quad  \max_{x\in \bbr}{J'_{m+\frac{1}{2}}(x)} \sim (m+\frac{1}{2})^{-\frac{2}{3}}. 
$$
 Hence,  we can conduct similar estimates as $I_i$ for $L_i,i=1,2,3$ to get the lower gradient bound and
\begin{equation*}
\frac{\left\|\nabla w_m^l\right\|_{L^\infty(\Sigma_\xi)}}{k} \geqslant C\frac{(m+1/2)^{13/6}}{(|l|+1)^{1/2}}  \geqslant C (m +\frac{1}{2})^{\frac{5}{3}} \rightarrow \infty  \quad \textit{as} \quad m \rightarrow \infty,
\end{equation*}
where $C$ is a constant independent of $m$ and $l$.
\end{proof}

\section*{Acknowledgment}
We are grateful to  Huaian Diao and Yan Jiang  for their insightful
comments during discussions about this work.

\newpage
\bibliographystyle{abbrv}
\bibliography{cite}


\end{document}